\documentclass[a4paper,12pt,leqno]{amsart}
\usepackage{geometry}
\usepackage[utf8]{inputenc}
\usepackage{amssymb,amsfonts,amsmath,amsthm,eucal,url,cite}
\usepackage{enumitem}
\usepackage{tikz-cd}
\usepackage{hyperref}
\usepackage{multicol}
\usepackage{slashed}

\newcommand{\RR}{\mathbb{R}}
\newcommand{\CC}{\mathbb{C}}
\newcommand{\QQ}{\mathbb{Q}}
\newcommand{\NN}{\mathbb{N}}
\newcommand{\ZZ}{\mathbb{Z}}

\def\Span{\mathrm{Span}}

\newtheorem{theorem}{Theorem}[section]
\newtheorem{prop}[theorem]{Proposition}
\newtheorem{cor}[theorem]{Corollary}
\newtheorem{lemma}[theorem]{Lemma}
\newtheorem{definition}[theorem]{Definition}

\theoremstyle{definition}

\theoremstyle{remark}
\newtheorem{rem}[theorem]{\bf Remark}
\newtheorem{example}[theorem]{\bf Example}

\title[The characteristic group of locally conformally product structures]{The characteristic group of locally conformally product structures}
\author{Brice Flamencourt}
\address[B. Flamencourt]{Universität Stuttgart, Institut für Geometrie und Topologie, Fachbereich Mathematik, Pfaffenwaldring 57, 70569 Stuttgart, Germany.}
\email{brice.flamencourt@mathematik.uni-stuttgart.de}


\begin{document}
\maketitle

\begin{abstract}
A compact manifold $M$ together with a Riemannian metric $h$ on its universal cover $\tilde M$ for which $\pi_1(M)$ acts by similarities is called a similarity structure.  In the case where $\pi_1(M) \not\subset \mathrm{Isom}(\tilde M, h)$ and $(\tilde M, h)$ is reducible but not flat, this is a Locally Conformally Product (LCP) structure. The so-called characteristic group of these manifolds, which is a connected abelian Lie group, is the key to understand how they are built. We focus in this paper on the case where this group is simply connected, and give a description of the corresponding LCP structures. It appears that they are quotients of trivial $\mathbb{R}^p$-principal bundles over simply-connected manifolds by certain discrete subgroups of automorphisms. We prove that, conversely, it is always possible to endow such quotients with an LCP structure.
\end{abstract}

\section{Introduction}

A {\em similarity structure} on a compact manifold $M$ is the data of a Riemannian metric $h$ on its universal cover $\tilde M$ for which the deck-transformations are similarities (also called homotheties). Similarity structures can be divided in two families: those coming from lifts of Riemannian metrics on $M$ to $\tilde M$, for which the deck-transformations are isometries, and those where at least one of these transformations is a strict similarity with ratio less than $1$. In the first case, the structure is said to be Riemannian and its study belongs to Riemannian geometry. The second case, in which we will be interested in this article, is a pure subject of conformal geometry. We then restrict ourselves to the non-Riemannian case.

Similarity structures are quite rigid from the point of view of the holonomy group of $(\tilde M, h)$, and until very recently only flat and irreducible examples were known. This observation together with a previous result from Gallot \cite{Gal} on Riemannian cones led Belgun and Moroianu \cite{BeMo} to formulate the conjecture that those where the only possibilities, a statement that they proved under an additional assumption on the existing-time of geodesics. However, it appeared that a third case can occur, closing the list of possible holonomies as shown by Matveev and Nikolayevsky \cite{MN,MN2} in the analytic case and by Kourganoff \cite{Kou} in the smooth case. The last possible family actually contains manifolds for whose $(\tilde M, h)$ is reducible and admits a de Rham decomposition with two factors. More precisely, $(\tilde M, h) = \RR^q \times (N, g_N)$ where $\RR^q$ is an Euclidean space and $(N, g_N)$ is irreducible and incomplete. The manifolds belonging to this family are called {\em Locally Conformally Product} manifolds (which will be shorten to LCP in the sequel). LCP manifolds are also characterized as the ones whose similarity structure has reducible holonomy and is a Riemannian product with one of the factors being a complete Euclidean space \cite{FlaLCP}, thanks to a classification of flat similarity structures by Fried \cite{Frie}.

One can give a different approach of LCP manifolds using conformal geometry, which is equivalent to the previous exposition. Indeed, a non-Riemannian similarity structure can be defined by endowing the compact manifold $M$ with a conformal structure $c$ and a connection $D$ on $M$ preserving the conformal class, such that $D$ is always locally but not globally the Levi-Civita connection of a metric in $c$. The lifted connection $\tilde D$ on $\tilde M$ is then globally the Levi-Civita of a metric $h$ in the lifted conformal class $\tilde c$, uniquely defined up to a multiplication by a constant. This point of view enlighten some similarities with the intensively studied Locally Conformally Kähler (LCK) manifolds. For this reason, one of the first known examples of LCP manifolds was a subfamily of the OT-manifolds \cite{OT}, where the algebraic number field used for the construction has exactly $2$ complex embeddings. In this case, one can define a Kähler potential on the universal cover of the manifold inducing a non-Riemannian similarity structure with reducible non-flat holonomy. This example was further studied in \cite{FlaLCP} where it was shown that all the OT-manifolds admit LCP structures.

The OT-manifolds are moreover the only known examples of LCP manifolds admitting a compatible LCK structure. However, they do not admit a Kähler metric belonging to the induced conformal class. It has been proved in \cite{BFM} that the conformal class of an LCP manifold does not contain a Kähler or an Einstein metric, illustrating again the strong rigidity of these structures.

One can then observe that for a compact manifold to admit a non-Riemannian similarity structure already imposes important constraints, but the LCP manifolds are even more restrictive, and only few examples were given until a previous work of the author \cite{FlaLCP}. It seems that, despite the variety of examples that one can construct, we always need the same basic ingredients, giving hope for the possibility of finding a general construction pattern providing a good understanding of these structures. In particular, all examples come from solvmanifolds, a track which was followed by Andrada, del Braco and Moroianu \cite{ABM} in order to describe the LCP manifolds arising from solvable unimodular Lie algebras up to dimension $5$. Our goal in this paper is to continue investigating, in order to obtain a general construction. For this, we will start from objects introduced in previous works and characterizing LCP structures.

In order to prove that the number of possibilities for the holonomy of $(\tilde M, h)$ is limited, Kourganoff \cite{Kou} used an important tool on LCP manifolds, which is the restriction $P$ of the fundamental group $\pi_1(M)$ to the non-flat part $N$. The closure $\bar P$ of this restriction is a subgroup of the group of similarities of $(N, g_N)$, and its identity component ${\bar P}^0$ is an abelian Lie group containing only isometries. The group $\RR^q \times {\bar P}^0$, where $\RR^q$ has to be understood as the translations of the Euclidean space, is called the {\em characteristic group}. It consists of isometric transformations of $(\tilde M, h)$ and encodes important information on the LCP structure. In particular it helps understanding the action of the fundamental group. For instance, using the action of $\pi_1(M)$ by conjugation on the characteristic group, it was shown in \cite{FlaLCP} that the similarity ratios of elements of $\pi_1(M)$ are units of an algebraic number field.

The characteristic group  being a connected abelian Lie group, it is a product between an Euclidean space and a torus. However, all examples provided so far have simply connected characteristic group, suggesting the significance of understanding this particular case. In this article, we then focus on giving a description of LCP manifolds with simply connected characteristic groups, using the remarkable fact that this group then acts freely and properly on the manifold $\tilde M$, thus implying the existence of a new decomposition of $\tilde M$ as a product $\RR^p \times C$ with $\RR^q \subset \RR^p$, no longer adapted to the metric, but allowing a nice understanding of the action of $\pi_1(M)$.

A glance at the known examples lets us expect that the deck-transformations admit the following description: there exists a discrete group $H$ acting freely, properly and co-compactly on the factor $C$ and a subgroup $\Omega$ of $\mathrm{GL}_p (\ZZ) \times H$ such that
\begin{equation}
\pi_1(M) = \ZZ^p \rtimes \Omega,
\end{equation}
where $\ZZ^p$ acts on $\RR^p$ by translations. However, the situation appears to be less ideal, and the group $\Omega$ can take values in the automorphisms of the trivial $(S^1)^p$-principal bundle over $C$, lifted to $\RR^p \times C$. The group $\pi_1(M)$ is then not always a semi-direct product. Moreover, the action of the group $\Omega$  on $C$ is not necessarily free, turning $C/\Omega$ into a good compact orbifold rather than a manifold (see Theorem~\ref{admissibledata}). It is then natural to ask if, given a good compact orbifold together with a suitable lift of its fundamental group to the automorphisms of a trivial torus bundle over its universal cover, one can in turn construct an LCP manifold. The answer is positive, as we show in Theorem~\ref{converse}. In this process, we also describe all the possible LCP structures by providing a way to construct the metrics for which the deck-transformations act by similarities. In the last section, we discuss new examples and open questions, showing why some results we could conjecture are actually false. We also give a necessary and sufficient condition for the construction of an LCP structure when $C/\Omega$ is a manifold (see Proposition~\ref{existencemorph}).

\section{Preliminaries} \label{preliminaries}

Let $M$ be a manifold endowed with a conformal structure, i.e. a set $c$ 	of Riemannian  metrics such that for all $g, g' \in c$, there exists $f : M \to \RR$ satisfying $g = e^{2f} g'$.

A Weyl connection on the conformal manifold $(M,c)$ is a a torsion-free connection $D$ such that $D$ preserves $c$ in the sense that for any $g \in c$, there exists a $1$-form $\theta_g \in \Omega^1 (M)$ with
\begin{equation}
D g = - 2 \theta_g \otimes g.
\end{equation}
In this case, $\theta_g$ is called the Lee form of $D$ with respect to $g$.

If there is $g \in c$ such that $\theta_g$ is closed, then $\theta_{g'}$ is closed for any metric $g' \in c$ and this is equivalent to $D$ being locally the Levi-Civita connection of a metric in $c$, which means that for any $x \in M$ there is an open set $U \subset M$ with $x \in U$ and a metric $g \in c$ such that the Levi-Civita connection of $g$ coincides with $D$ on $U$. The Weyl connection $D$ is then said to be {\em closed}. This statement holds globally if and only if there is a metric $g  \in c$ such that $\theta_g$ is exact and in this case, $D$ is said to be {\em exact}.

We recall that a similarity (also called a homothety) between two Riemannian manifolds $(M_1, g_1)$, $(M_2, g_2)$ is a diffeomorphism $\varphi: M_1 \to M_2$ satisfying
\begin{equation}
\varphi^* g_2 = \lambda^2 g_1
\end{equation}
for some $\lambda > 0$ called {\em the similarity ratio of} $\varphi$. A similarity structure on $M$, as defined in \cite{Kou}, is a metric $h$ on the universal cover $\tilde M$ such that $\pi_1(M)$ acts by similarities on $(M, h)$. When $\pi_1(M)$ acts only by isometries, this similarity structure is called Riemannian. If the Weyl connection $D$ is closed, its lift $\tilde D$ to the universal cover $\tilde M$ of $M$ together with the lifted conformal class $\tilde c$ is an exact Weyl connection, and there exists a metric $h_D \in \tilde c$, unique up to multiplication by a constant, such that $\tilde D$ is the Levi-Civita connection of $h_D$. The fundamental group $\pi_1(M)$ acts on $(\tilde M, h_D)$ by similarities, defining a similarity structure on $M$. As it was discussed in \cite{FlaLCP}, there is a one-to-one correspondence between similarity structures up to constant multiplication and closed Weyl structure on $M$. Through this identification, Riemannian similarity structures correspond to exact Weyl structures.

From now on, we consider a compact conformal manifold $(M, c)$ endowed with a closed, non-exact Weyl connection $D$, which is non-flat and has reducible holonomy. The structure $(M, c, D)$ is called a Locally Conformally Product structure (or LCP for short) \cite{FlaLCP}. A theorem due to Kourganoff \cite[Theorem 1.5]{Kou} allows one to understand the structure of these LCP manifolds by looking at the Riemannian metric $h_D$ induced on $\tilde M$ by $D$.

\begin{theorem}[Kourganoff] \label{LCPTh}
Let $D$ be a closed, non-exact Weyl structure on a compact conformal manifold $(M, c)$. Assume moreover that $D$ is non-flat and has reducible holonomy. Then, there exists $q \ge 1$ and an irreducible incomplete Riemannian manifold $(N, g_N)$ such that the universal cover $\tilde M$ of $M$ endowed with the metric $h_D$ induced by $D$ is isometric to the Riemannian product $\RR^q \times (N, g_N)$.
\end{theorem}

It was proved in \cite{FlaLCP} that $(M,c,D)$ is an LCP structure if and only if $(\tilde M, D)$ is reducible and has a Riemannian factor which is an Euclidean space. This fact will be often used in order to show that the examples we will give are indeed LCP manifolds.

According to Theorem~\ref{LCPTh}, one has $(\tilde M, h_D) \simeq \RR^q \times (N, g_N)$, where $\RR^q$ and $(N, g_N)$ will respectively be called {\em the flat part} and the {\em non-flat part} of $\tilde M$. This can be reformulated by saying that $(\tilde M, h_D)$ has a de Rham decomposition with 2 factors, and since $\pi_1(M)$ preserves the connection $\tilde D$, it must preserves the de Rham decomposition up to the order of the factors, but it must also preserves the flat factor $\RR^q$. Consequently, $\pi_1(M)$ preserves the de Rham decomposition, and any $\gamma \in \pi_1(M)$ can be written as $\gamma =: (\gamma_E, \gamma_N)$ where $\gamma_E$ and $\gamma_N$ are similarities of $\RR^q$ and $N$ respectively. We then define $P := \lbrace \gamma_N \ \vert \ \gamma \in \pi_1(M) \rbrace$, and we denote by $\bar P$ the closure of $P$ in $\mathrm{Sim} (N, g_N)$ with respect to the compact-open topology.

The connected component of the identity in $\bar P$ is denoted by ${\bar P}^0$. It has been shown in \cite[Lemma 4.1]{Kou} that ${\bar P}^0$ is an abelian Lie group, satisfying ${\bar P}^0 \subset \mathrm{Iso} (N, g_N)$. In particular, it is the product of a real vector space and a torus.

The decomposition $\tilde M = \RR^q \times N$ induces a natural foliation $\tilde {\mathcal F}$ on $\tilde M$ via the submersion $\tilde M \to N$, whose leaves are the sets $\tilde {\mathcal F}_x := \RR^q \times \{ x \}$ for $x \in N$. In turn, the foliation $\tilde {\mathcal F}$ induces a foliation $\mathcal F$ on the compact manifold $M$. The closures of the leaves of $\mathcal F$ define a singular Riemannian foliation $\bar{\mathcal F}$ on $M$ \cite[Theorem 1.9]{Kou}, and if one denotes by $\pi_M : \tilde M \to M$ the canonical projection, the leaves of $\bar {\mathcal F}$ are exactly the sets $\pi_M (\RR^q \times {\bar P}^0 x)$ for $x \in N$ \cite[Lemma 4.11]{Kou}. In view of this last property we define for $x \in N$ the set $\tilde {\mathcal{CF}}_x := \RR^q \times {\bar P}^0 x$, so that $\tilde {\mathcal{CF}} = \{\tilde {\mathcal{CF}}_x \ \vert \ x \in N \}$ is a singular foliation of $\tilde M$. Since ${\bar P}^0$ is an abelian Lie group which acts by isometries on $(N, g_N)$, the Riemannian manifolds $\RR^q \times {\bar P}^0 x$, $x \in N$, with the metric induced by $g_N$ are products of an Euclidean space with a flat torus, as it was noticed in \cite{Kou} and \cite{FlaLCP}, however their dimensions depend on $x$ a priori.

We call a {\em lattice} of an abelian Lie group $G$ any discrete subgroup $H$ of $G$. If $G/H$ is compact, then $H$ is called a {\em full lattice}. It was shown in \cite[Lemma 4.18]{Kou} that the group
\begin{equation}
\Gamma_0 := \pi_1(M) \cap (\mathrm{Sim} (\RR^q) \times {\bar P}^0)
\end{equation}
is a full lattice of $\RR^q \times {\bar P}^0$. In particular, it is an abelian subgroup of $\mathrm{Isom}(\tilde M, h_D)$. In addition, it was shown in \cite[Lemma 2.10]{FlaLCP}, by adapting the incorrect proof of \cite[Lemma 4.17]{Kou}, that $P$ is isomorphic to $\pi_1(M)$. Since ${\bar P}^0$ is a normal subgroup of $\bar P$,  being its identity component, $\Gamma_0$ is a normal subgroup of $\pi_1(M)$ by definition.

We will study in this article the case where ${\bar P}^0 \simeq \RR^{p-q}$ for some $p \ge q$, in order to give a precise description of these manifolds.

\section{General results} \label{SecGeneral}

\subsection{Properties of some group actions on LCP manifolds}
Let $(M, c, D)$ be an LCP manifold. We keep the notations of the preliminary section. The group $\Gamma_0$ is a finitely generated abelian group since it is a lattice of $\RR^q \times {\bar P}^0$, thus $\Gamma_0 = \Gamma_0^{tor} \oplus \ZZ^p$ where $\Gamma_0^{tor}$ is the torsion subgroup of $\Gamma_0$ and $p \ge q$. Our first goal in this section is to give a special decomposition of the universal cover $\tilde M$ of $M$. For this purpose, we consider a representative $\Gamma_0^F$ of $\Gamma_0 / \Gamma_0^{tor} \simeq \ZZ^p$ in $\Gamma_0$ with basis $(\gamma_1, \ldots, \gamma_p)$. We denote by "$\exp$" the exponential map of the Lie group $\RR^q \times {\bar P}^0$, and we consider $(X_1, \ldots, X_p) \in (T_e (\RR^q \times {\bar P}^0))^p$, where $e$ is the neutral element, such that $\gamma_i = \exp (X_i)$ for any $1 \le i \le p$. The subgroup $\exp^{-1}(\Gamma_0)$ is a full lattice of $T_e (\RR^q \times {\bar P}^0)$ and one easily sees that $\exp^{-1}(\Gamma_0) = \langle X_1, \ldots, X_p \rangle \oplus \exp^{-1}(\Gamma_0^{tor})$, thus $T_e (\RR^q \times {\bar P}^0) = \mathrm{Span}(X_1, \ldots, X_p) \oplus \mathrm{Span}(\exp^{-1}(\Gamma_0^{tor}))$.

We define the subgroup $F$ of $\RR^q \times {\bar P}^0$ by
\begin{equation} \label{defF}
F := \exp (\mathrm{Span} (X_1, \ldots, X_p)).
\end{equation}
We claim that $\exp : \mathrm{Span} (X_1, \ldots, X_p) \to F$ is an isomorphism. Indeed, it is sufficient to prove that the map is into. But for any $X \in \mathrm{Span} (X_1, \ldots, X_p)$ such that $\exp (X) = 0$, one has $X \in \exp^{-1}(\Gamma_0^{tor}) \cap \mathrm{Span}(X_1, \ldots, X_n) = \{ 0 \}$, so $X = 0$ and we deduce the injectivity.

In addition, $F$ is a closed subgroup. Indeed, if we write $\RR^q \times {\bar P}^0$ as the product $\RR^p \times (S^1)^m$ of an Euclidean space and a torus, the projection of the basis $(\gamma_1, \ldots, \gamma_p)$ onto $\RR^p$ is a basis of $\RR^p$ because otherwise $\Gamma_0$ would not be a full lattice of $\RR^q \times {\bar P}^0$. From the point of view of the Lie algebra, it implies that the projection of $(X_1, \ldots, X_p)$ onto the Lie algebra of $\RR^p$ is a basis. We then easily define a continuous bijection between $\RR^p$ and $\mathrm{Span} (X_1, \ldots, X_p) \simeq F$, proving that $\RR^p \simeq F$ and that $F$ is closed.

The group $F$ thus represents the non-compact part of the group $\RR^q \times {\bar P}^0$. In order to show that $F$ acts freely on $\tilde M$, we first prove two technical lemmata:

\begin{lemma} \label{techlattice}
Let $\Gamma \subset \Gamma'$ be two full lattices of $\RR^p$. Then, for any $\gamma \in \Gamma'$ there exists $r \ge 1$ such that $\gamma^r \in \Gamma$.
\end{lemma}
\begin{proof}
The space $\RR^p/\Gamma$ is a covering of $\RR^p/\Gamma'$ with fiber $\Gamma'/\Gamma$. This is a finite covering because both spaces are compact, thus $\Gamma'/\Gamma$ is finite and each one of its elements has finite order, proving the lemma.
\end{proof}

\begin{lemma} \label{tech2}
Let $\mathcal M$ be a smooth manifold on which acts the group $\RR^p$. Assume that a full lattice of $\RR^p$ acts freely and properly on $\mathcal M$. Then, $\RR^p$ acts freely on $\mathcal M$.
\end{lemma}
\begin{proof}
We denote by $\Gamma$ the full lattice of the statement. Let $x \in \mathcal M$ and consider the set $\RR^p \cdot x$. Let $S(x) = \{ a \in \RR^p \ \vert \ a \cdot x = x \}$ be the stabilizer of $x$ in $\RR^p$. Then one has $\RR^p \cdot x \simeq \RR^p/S(x)$. We want to prove that $S(x)$ only contains the identity.

By definition, the free abelian group $\Gamma$ is a full lattice of $\RR^p$ and $\Gamma$ acts freely and properly on $\mathcal M$. Since $\Gamma$ stabilizes $\RR^p \cdot x$, one has that $\Gamma$ acts freely and properly on $\RR^p \cdot x$, so it is a full lattice of the abelian group $\RR^p/S(x)$. A lattice of an abelian lie group has a rank lower than the dimension of the group, so $p \le \mathrm{dim} (\RR^p/S(x))$, and moreover $\mathrm{dim} (\RR^p/S(x)) \le p$ because $\mathrm{dim} (\RR^p) = p$, thus $\mathrm{dim} (\RR^p/S(x)) = p$ and $S(x)$ is a discrete subgroup of $\RR^p$.

One has $(\RR^p\cdot x) / \Gamma \simeq \RR^p/\langle S(x), \Gamma \rangle$, where $\langle S(x), \Gamma \rangle$ is the group generated by $S(x)$ and $\Gamma$. In particular $\langle S(x), \Gamma \rangle$ is a full lattice of  $\RR^p$.

We now pick $a \in S(x)$. By Lemma~\ref{techlattice} applied to the full lattices $\Gamma \subset \langle S(x), \Gamma \rangle$, there exists $r \ge 1$ such that $a^r \in \Gamma$. Since $a^r \cdot x = x$ and $\Gamma$ acts freely, one has that $a^r = \mathrm{id}$. In addition, $S(x)$ has no torsion because it is a subgroup of $\RR^p$ thus $a = \mathrm{id}$ and $S(x) = \{ \mathrm{id} \}$.
\end{proof}

\begin{cor} \label{Factsfree}
The group $F$ acts freely on $\tilde M$.
\end{cor}
\begin{proof}
We apply Lemma~\ref{tech2} to the action of $F \simeq \RR^p$ on $\tilde M$, knowing that the full lattice $\Gamma_0^F$ of $F$ acts freely and properly on $\tilde M$ because $\Gamma_0^F \subset \Gamma_0 \subset \pi_1(M)$.
\end{proof}

\begin{cor} \label{PSimCon}
If ${\bar P}^0$ is simply connected, then $\RR^q \times {\bar P}^0$ acts freely on $\tilde M$. In particular, ${\bar P}^0$ acts freely on $N$.
\end{cor}
\begin{proof}
If ${\bar P}^0$ is simply connected, the group $F$ defined in equation~\ref{defF} is equal to the whole group $\RR^q \times {\bar P}^0$, and by Corollary~\ref{Factsfree} it acts freely on $\tilde M$.
\end{proof}

We now discuss a problem tackled in \cite[Lemma 4.16]{Kou}, whose proof was shown to be incorrect. We would like to understand the action of the group $P$ on $N$. As we explained before, it is known that $P$ is isomorphic to $\pi_1(M)$. However, we don't know if the action of $P$ is free. We give here a first result which states that this is true if and only if the restricted action of $\Gamma_0$ is free. 

\begin{prop} \label{actionrestricted}
The group $P$ acts freely on $N$ if and only if the restriction of $\Gamma_0$ to $N$, which coincides with $P \cap {\bar P}^0$, acts freely.
\end{prop}
\begin{proof}
If $P$ acts freely on $N$, then it is easily seen that the restriction of $\Gamma_0$ to $N$ acts freely. It remains to prove the converse.

Assume that there is $\gamma \in \pi_1(M) \setminus \{\mathrm{id}\}$ and $x \in N$ such that $\gamma_N (x) = x$ (we recall that $\gamma_N$ is the part of $\gamma$ acting on $N$). The transformation $\gamma$ stabilizes the closed leaf $\tilde{\mathcal{CF}}_x$. Since $\pi_1(M)$ acts freely and properly discontinuously on $\tilde M$, the group $\pi_1(M)/\Gamma_0$ acts freely and properly discontinuously on $\tilde M/\Gamma_0$ (see \cite[Proposition 4.1]{FlaLCP} for example). We denote by $\bar\gamma$ the equivalence class of $\gamma$ in $\pi_1(M)/\Gamma_0$. As $\bar\gamma$ stabilizes the compact set $\tilde{\mathcal{CF}}_x/\Gamma_0$ (which is compact because $\Gamma_0$ is a full lattice of $\RR^q \times {\bar P}^0$ and $\tilde{\mathcal{CF}}_x = \RR^q \times {\bar P}^0 x$), the set $\{ \bar\gamma^m (0,x) \ \vert \ m \in \NN \}$ is finite because $\pi_1(M)/\Gamma_0$ acts properly on $\tilde M / \Gamma_0$. Thus there exists $m \in \NN \setminus \{0\}$ such that $\bar\gamma^m (0,x) = (0,x)$, but $\pi_1(M)/\Gamma_0$ acts freely, so $\bar\gamma^m = \mathrm{id}$ and $\gamma^m \in \Gamma_0$.

The restriction $\gamma_E$ of $\gamma$ to the flat part $\RR^q$ is of the form $\RR^q \ni a \mapsto A a + b$ where $A \in \mathrm{GL}_q(\RR)$ and $b \in \RR^q$. Since $\gamma$ has no fixed point and $\gamma_N (x) = x$, $\gamma_E$ has no fixed point. One has that $\gamma_E^m$ is a translation because $\gamma^m \in \Gamma_0$, so $A^m = I_q$, and considering the polynomial $X^m - 1 =: (X-1)R(X)$ one has $\RR^q = \ker (A - I_q) \oplus \ker R(A) =: V_1 \oplus V_2$. According to this decomposition, $b =: b_1 + b_2$, and $(I_q - A)\vert_{V_2}$ being invertible there exists a unique $v_2 \in V_2$ such that $(I_q - A) v_2 = b_2$.  One then has $\gamma_E (v_2)= v_2 + b_1$, implying that $b_1 \neq 0$. Hence $\gamma^m (v_2 , x) = (v_2 + m b_1, x)$ giving that $\gamma^m$ is a non-trivial element of $\Gamma_0$. Since $\pi_1(M)$ is isomorphic to $P$, $(\gamma^m)_N$ is a non-trivial element of $P$, which is the restriction of an element of $\Gamma_0$ to $N$ and which has $x$ as a fixed point.
\end{proof}

\begin{cor} \label{freeonN}
If ${\bar P}^0$ is simply connected, then $P$ acts freely on $N$.
\end{cor}
\begin{proof}
Assume that ${\bar P}^0$ is simply connected. By Corollary~\ref{PSimCon}, ${\bar P}^0$ acts freely on $N$. In particular, the restriction of $\Gamma_0$ to $N$, which is contained in ${\bar P}^0$, acts freely and Proposition~\ref{actionrestricted} implies that $P$ acts freely on $N$.
\end{proof}

The results of Corollaries~\ref{PSimCon} and \ref{freeonN} motivate the following definition:

\begin{definition}
The group $\RR^q \times {\bar P}^0$ will be called the {\em characteristic group} of the LCP manifold.
The LCP structure $(M,c,D)$ is said to be {\em simple} if its characteristic group is simply connected.
\end{definition}

\begin{rem} \label{remark1} The action of the group $\pi_1(M)$ on $\tilde M$ descends to an action of $\pi_1(M) / \Gamma_0$ on $N/{\bar P}^0 \simeq \tilde M / (\RR^q \times {\bar P}^0)$ because $\RR^q \times \bar P$ is stable under conjugation by elements of $\pi_1(M)$. By the proof of Proposition~\ref{actionrestricted}, we know that this action is proper. Indeed, if $K$ is a compact subset of $N/{\bar P}^0$, the set $E := \{ \bar\gamma \in \pi_1(M)/\Gamma_0 \ \vert \ (\bar\gamma K) \cap K \neq \emptyset \}$ is equal to $\{ \bar\gamma \in \pi_1(M)/\Gamma_0 \ \vert \ (\bar\gamma K') \cap K' \neq \emptyset \}$ where $K'$ is the inverse image of $K$ by the projection $\tilde M /\Gamma_0 \to N/{\bar P}^0$, and $K'$ is compact because $(\RR^q \times {\bar P^0})/ \Gamma_0$ is compact, thus the set $E$ is finite.
\end{rem}

We finally prove a useful property, which can be used to identify the characteristic group in some special situations.

\begin{prop} \label{density}
The image of $\RR^q$ in the torus $(\RR^q \times {\bar P}^0)/\Gamma_0$ is dense.
\end{prop}
\begin{proof}
Let $(a, p) \in \RR^q \times {\bar P}^0$ and let $U$, $V$ be neighborhoods of $a$ and $p$ respectively. Since $P \cap {\bar P}^0 = \Gamma_0 \vert_N$ is dense in ${\bar P}^0$, there exists $\gamma_0 \in \Gamma_0$ such that $\gamma_0 \in \RR^q \times V$, and then one can find $a' \in \RR^q$ such that $a' + \gamma_0 \in U \times V$. This implies that $\langle \RR^q, \Gamma_0 \rangle$ is dense in $\RR^q \times {\bar P}^0$, and thus the image of $\RR^q$ in $(\RR^q \times {\bar P}^0)/\Gamma_0$ is dense.
\end{proof}

\subsection{Finite coverings of LCP manifolds} \label{finitecovering}
In \cite[Theorem 1.10]{Kou}, it was shown using Selberg's lemma that there exists a finite covering $M'$ of $M$ such that the closures of the leaves of the foliation $\mathcal F'$ induced by $\mathcal F$ on $M'$ are flat tori. In particular, $M'$ is still an LCP manifold. We recall here the key point of the proof for the convenience of the reader, and also because we will use it later.

\begin{prop} \label{notorsion}
Let $(M, c, D)$ be an LCP structure. Then, up to a finite covering of $M$, the action of $\pi_1(M)$ on the characteristic group by conjugation has no torsion.
\end{prop}
\begin{proof}
The action of $\pi_1(M)$ on the characteristic group by conjugation is well-defined, because ${\bar P}^0$ is a normal subgroup of ${\bar P}$, and the restriction of $\pi_1(M)$ to $\RR^q$ contains only similarities which are in particular affine maps of the form $\mathcal A := \RR^q \ni a \mapsto R a + t$ with $(R,t) \in \mathrm{GL}_q (\RR) \times \RR^q$. For any $b \in \RR^q$, if we define the translation $\tau_b : \RR^q \ni a \mapsto a + b$, one has $\mathcal A \tau_b \mathcal A^{-1} = (\RR^q \ni a \mapsto a + R b) =: \tau_{R b}$.

We then define $J : \pi_1(M) \to \mathrm{Aut}(\RR^q \times {\bar P}^0)$, $\gamma \mapsto (\gamma_0 \mapsto \gamma \gamma_0 \gamma^{-1})$. For all $\gamma \in \pi_1(M)$, $J(\gamma)$ preserves $\Gamma_0$, which is a full-lattice of $\RR^q \times {\bar P}^0$. This implies that $J(\gamma)$ descends to an automorphism of the torus $(\RR^q \times {\bar P}^0)/\Gamma_0$, and $J(\gamma)$ defines a unique linear transformation $\tilde J(\gamma)$ on the universal cover $\RR^p$ of $\RR^q \times {\bar P}^0$, which preserves the full-lattice given by the lift of $\Gamma_0$. Up to a linear transformation, one can assume that this lattice is the canonical lattice $\ZZ^p$ of $\RR^p$. It follows that $\tilde J (\gamma) \in \mathrm{GL}_p(\RR)$. The map $\tilde J : \gamma \mapsto \tilde J (\gamma)$ is then a group homomorphism, and by Selberg's lemma there exists a subgroup $G$ of $\tilde J(\pi_1(M))$ with finite index and without torsion element. Thus $\tilde J^{-1} (G)$ is a subgroup of $\pi_1 (M)$ of finite index, such that $J (\tilde J^{-1} (G))$ has no torsion element.
\end{proof}

The question of wether this finite cover $M'$ can always be taken to be $M$ was raised, since this result is true when $\mathcal F$ is of dimension $1$. Here we answer negatively by providing a counter-example.

\begin{example}
Let $\tilde M := \RR^4 \times \RR_+^*$. We consider the matrices
\begin{align*}
A := \left( \begin{matrix} 1 & 2 & 0 & 0 \\ 2 & 3 & 0 & 0 \\ 0 & 0 & 1 & 2 \\ 0 & 0 & 2& 3  \end{matrix} \right), && B := \left( \begin{matrix} 1 & 0 & 0 & 0 \\ 0 & 1 & 0 & 0 \\ 0 & 0 & -1 & 0 \\ 0 & 0 & 0 & -1  \end{matrix} \right).
\end{align*}
The matrices $A$ and $B$ commutes so they can be diagonalized in a common basis $(X_1, Y_1, X_2, Y_2)$ where $(X_1, Y_1) \in (\RR^2 \times \{0\}^2)^2$ and $(X_2, Y_2) \in (\{0\}^2 \times \RR^2)^2$. In this basis, $A$ is written as $\mathrm{Diag} (\lambda, -\lambda^{-1}, \lambda, -\lambda^{-1})$ with $\lambda = 2 + \sqrt{5} > 1$ and $B$ remains the same. We define a group $G$ of transformations of $\tilde M$ by
\begin{equation}
G := \langle \ZZ^4, T_A : (a,t) \mapsto (A a, \lambda t), T_B : (a,t) \mapsto (B a + (0, 1/2, 0,0)^T, t) \rangle,
\end{equation}
where $\ZZ^4$ is the standard lattice acting on $\RR^4$. Simple computations provide
\begin{align} \label{relationTAB}
T_A \circ T_B (a,t) = T_B \circ T_A (a, t) + ((1,1,0,0)^T, 0), && T_B^2 (a,t) = (a + (0,1,0,0)^T, t).
\end{align}
Let now $(a,t) \in \tilde M$ and $g \in G$ such that $g (a,t) = (a,t)$. According to the relations \eqref{relationTAB} one has that there exist $\delta \in \{0,1\}$, $n \in \ZZ$ and $Z \in \ZZ^4$ such that
\begin{equation} \label{ecritureg}
g = Z \circ T_B^\delta \circ T_A^n,
\end{equation}
implying
\begin{equation}
(a,t) = g(a,t) = (B^\delta A^n a + (0, \delta/2,0,0)^T + Z, \lambda^n t).
\end{equation}
From the identity $t  = \lambda^n t$ it follows that $n = 0$, and $B^\delta a + (0, \delta/2,0,0)^T + Z = a$ gives $Z = 0$ and $\delta = 0$, so we conclude that $g = \mathrm{id}$ and that $G$ acts freely. In order to prove that $G$ acts properly discontinuously, it is sufficient to notice that $\tilde M/\ZZ^4 \simeq (S^1)^4 \times \RR_+^*$ and $G/\ZZ^4$ acts properly on $(S^1)^4 \times \RR_+^*$ since the class of $T_A$ acts by multiplication by $\lambda$ on the coordinate $t$ and $T_B$ has order $2$. Consequently, $M := \tilde M/G$ is a manifold, and it is compact since $G([0,1]^4 \times [1,\lambda]) = \tilde M$.

We consider the Riemannian metric $h$ on $\tilde M$ given by
\begin{equation}
h := dx_1^2 + dx_2^2 + t^4 (dy_1^2 + dy_2^2) + dt^2
\end{equation}
where $(x_1,y_1,x_2,y_2)$ are the coordinates in the basis $(X_1, Y_1,X_2,Y_2)$ of $\RR^4$ and $t$ is the parameter of the factor $\RR_+^*$. For any $g \in G$, if we use the decomposition $g = Z \circ T_B^\delta \circ T_A^n$ given by \eqref{ecritureg}, we have $g^* h := \lambda^{2n} h$, meaning that $G$ acts on $\tilde M$ by similarities which are not all isometries. The Levi-Civita connection of $h$ thus descends to a connection $D$ on $M$ while $h$ induces a conformal class $c$ on $M$. The triple $(M,c,D)$ is an LCP structure. From the point of view of the universal cover $\tilde M$, the flat part of this LCP structure is identified to the subspace of $\RR^4$ given by $\mathrm{Span}(X_1,X_2)$, and its non-flat part is identified to the manifold $\mathrm{Span} (Y_1,Y_2) \times \RR_+^*$.

It is easily seen that the group ${\bar P}^0$ in this case consists of all the translations lying in $\mathrm{Span} (Y_1,Y_2)$, and thus for any $(y_1,y_2,t)$ in the non-flat part, one has $\tilde{\mathcal{CF}}_{(y_1,y_2,t)} = \RR^4 \times \{ t \}$. We deduce that the closed leaf $\pi_M (\tilde{\mathcal{CF}}_{(y_1,y_2,t)})$ is diffeomorphic to $\RR^4 \times \{t \}/ S$ (see \cite[Lemma 3.5]{FlaLCP} for additional details), with $S := \{ g \in G \ \vert \ \forall (x_1,x_2) \in \RR^2, \ g(x_1,x_2,y_1,y_2,t) \in \RR^4 \times \{t \} \} = \langle \ZZ^4, T_B \rangle$, where the last equality comes from the decomposition \eqref{ecritureg}. Since $S$ acts freely and properly discontinuously on $\RR^4 \times \{ t \} \simeq \RR^4$, one has $\pi_1(\pi_M (\tilde{\mathcal{CF}}_{(y_1,y_2,t)})) \simeq S$, but $S$ is not abelian, so $\pi_M (\tilde{\mathcal{CF}}_{(y_1,y_2,t)})$ is not a torus.
\end{example}

We now come back to the finite cover $M'$ of $M$ introduced before. From the proof of \cite[Theorem 1.10]{Kou} already cited above, we have the following property: denoting by ${\bar {P'}}^0$ the equivalent of the group ${\bar P}^0$ for $M'$ (since $M'$ is also an LCP manifold), one has ${\bar {P'}}^0 = {\bar P}^0$ and the action of $\pi_1(M') \subset \pi_1(M)$ on $\RR^q \times {\bar {P'}}^0$ by conjugation has no torsion. This motivates the following definition:

\begin{definition}
We say that the LCP structure $(M, c, D)$ is {\em torsion-free} if the action of $\pi_1(M)$ by conjugation on the characteristic group has no torsion.
\end{definition}

\section{Simple LCP manifolds}

In this section we give a description of simple LCP manifolds (i.e. with simply connected characteristic group). In a first part, we study the structure of such manifolds to derive the existence of necessary conditions, and we prove in a second part that any manifold satisfying these conditions can be endowed with an LCP structure.

\subsection{Analysis of the LCP structure} \label{analysis}

Let $(M, c, D)$ be an LCP manifold. We keep the notations of the preliminary section, and we assume that ${\bar P}^0 \simeq \RR^{p-q}$ for some $p \ge q$, i.e. that the LCP structure is simple and its characteristic group is isomorphic to $\RR^p$. By Corollary~\ref{PSimCon} this group acts freely on $\tilde M$. Moreover, we obtain the following result:

\begin{prop} \label{leafiso}
If $(M,c,D)$ is simple, then for any $x \in N$, the Riemannian manifold $\tilde {\mathcal{CF}}_x = \RR^q \times {\bar P}^0 x$ is isometric to the Euclidean space $\RR^p$.
\end{prop}
\begin{proof}
Let $x \in N$. By Corollary~\ref{PSimCon}, the group ${\bar P}^0$ acts freely on $N$, so $\tilde {\mathcal{CF}}_x = \RR^q \times {\bar P}^0 x \simeq \RR^q \times {\bar P}^0 \simeq \RR^p$. Moreover, we recall that ${\bar P}^0$ is a subgroup of the isometries of $(N,g_N)$, hence the metric $h_D$ restricted to $\tilde {\mathcal F}_x$ is invariant by the group $\RR^q \times {\bar P}^0 \simeq \RR^p$, which implies that $\tilde {\mathcal{CF}}_x$ is an Euclidean space isometric to $\RR^p$.
\end{proof}

Corollary~\ref{PSimCon} gives that the group ${\bar P}^0$ acts freely on $N$. It has been shown in \cite[Lemma 4.9]{Kou} that $\bar P$ acts properly on $N$, and so does ${\bar P}^0$ as a closed subgroup of $\bar P$. Consequently, by \cite[Theorem 21.10]{Lee} the quotient $C := N/{\bar P}^0$ is a smooth manifold. The canonical submersion $\pi_N : N \to C$ is a Riemannian submersion because ${\bar P}^0$ acts by isometries on $(\tilde M, h_D)$, and it defines a $\RR^{p-q}$-principal bundle. Here, we do not need to specify whether the action of $\RR^{p-q}$ is on the right or the left because the group $\RR^{p-q}$ is abelian. In particular, the fibers of the principal bundle $\pi : \tilde M \to C$ are contractible and by \cite[Corollary 29.3]{Stee} this bundle admits a section, so it is trivial because it is a principal bundle. Thus we can write $N \simeq \RR^{p-q} \times C$.

The metric $g_N$ is invariant under the action of the group ${\bar P}^0 \simeq \RR^{p-q}$ (which consists of isometries of $g_N$), thus it is a $\RR^{p-q}$-invariant metric of the principal bundle $\RR^{p-q} \times C \to C$.

We now study the action of $\pi_1(M)$.  One has $\tilde M \simeq \RR^q \times N \simeq \RR^p \times C$ by the previous analysis, and since $\Gamma_0$ is a full-lattice of $\RR^q \times {\bar P}^0 \simeq \RR^p$ it induces a full-lattice on each fiber of the $\RR^p$-principal bundle $\tilde M \to C = N/{\bar P}^0$. Let us fix a basis of $\Gamma_0$. Using this basis, we can now identify $\RR^p$ with $\RR^q \times {\bar P}^0$ in such a way that $\Gamma_0$ is the canonical lattice $\ZZ^p$ of $\RR^p$. Under this identification, $\RR^q$ and ${\bar P}^0$ are vector subspaces $E^q$ and $E^{p-q}$ of $\RR^p$ satisfying $\RR^p = E^q \oplus E^{p-q}$. We will from now on identify $\Gamma_0$ with the canonical full-lattice $\ZZ^p$ of $\RR^p$ and $\tilde M$ with $\RR^p \times C$.

For any $a \in \RR^p$, we denote by $\tau_a$ the action of $a$ on $\tilde M$. We recall that $\gamma$ acts by conjugation on $\RR^p$. Indeed, the action of $\gamma_N$ by conjugation on $E^{p-q} \simeq {\bar P}^0$ is well-defined, because ${\bar P}^0$ is a normal subgroup of $\bar P$, so it is stable by the action of $P$ by conjugation. This transformation is an automorphism of the group $E^{p-q}$, and in particular it is a linear invertible map of $E^{p-q}$ viewed as a vector space. Moreover, $\gamma_E$ acts on $E^q$ as an affine transformation $E^{p-q} \ni v \mapsto R_\gamma v + t_\gamma$, where $(R_\gamma, t_\gamma) \in \mathrm{GL} (E^q) \times E^q$. Consequently, if for any $u \in E^q$ we denote by $t(u)$ the translation by $u$ in $E^q$, one has $\gamma_E t(v) \gamma_E^{-1} = t(R_\gamma v)$. Altogether, there exists a matrix $A_\gamma \in \mathrm{GL}_p(\RR)$ such that for any $a \in \RR^p$ one has $\gamma \tau_a \gamma^{-1} = \tau_{A_\gamma a}$, and $A_\gamma$ preserves the decomposition $\RR^p = E^q \oplus E^{p-q}$.

The matrix $A_\gamma$ stabilizes $\Gamma_0 \subset \RR^p$ because it is a normal subgroup of $\pi_1(M)$, and so does $(A_\gamma)^{-1} = A_{\gamma^{-1}}$, thus it is an element of $\mathrm{GL}_p(\ZZ)$. In addition, the transformation $\gamma$ on $\tilde M$ descends to a transformation $\bar \gamma$ on $N/{\bar P}^0$, because for any $(a,y) \in \RR^p \times N$ one has $\gamma \tau_a (0,y) = (\gamma \tau_a \gamma^{-1}) \gamma (0,y)$, thus $\gamma \tau_a (0,y)$ and $\gamma (0,y)$ are in the same coset modulo ${\bar P}^0$ since $\gamma \tau_a \gamma^{-1} = \tau_{A_\gamma a} \in \RR^p$. We deduce that for any $(a, x) \in \RR^p \times C$, one has
\begin{equation} \label{decompogamma}
\gamma (a,x) = \gamma \tau_a \gamma^{-1} \gamma (0,x) = \tau_{A_\gamma a} \gamma (0,x) = (A_\gamma a + f_\gamma (x), \bar \gamma (x)),
\end{equation}
where $f_\gamma : C \to \RR^p$ is a function whose projection on $\RR^q$ is constant, because $\gamma$ preserves the product structure $\tilde M \simeq \RR^q \times N$. Consequently, $\gamma$ is an automorphism of the trivial $\RR^p$-principal bundle $\RR^p \times C \to C$ (see \cite[Section 5]{KN1} for the definition). Moreover, since $A \in \mathrm{GL}_p(\ZZ)$, the map $\gamma$ descends to an automorphism of the trivial $(S^1)^p$-principal bundle $(\RR^p / \ZZ^p) \times C \simeq (S^1)^p \times C \to C$. The only elements of $\pi_1(M)$ which descend to the identity this way are the elements of $\ZZ^p \simeq \Gamma_0 \subset \pi_1(M)$, thus $\pi_1(M)/ \Gamma_0$ is a subgroup of the automorphisms of $(S^1)^p \times C \to C$.

We introduce the following definitions:

\begin{definition} \label{linpart}
For any principal bundle $P \to B$, we denote by $\mathrm{Aut}(P \to B)$ the set of  its automorphisms.

An automorphism of the trivial $\RR^p$-principal over $C$ can be written as a map $\RR^p \times C \ni (a,x) \mapsto (A a + f(x), \varphi(x))$ where $A \in \mathrm{GL}_p (\RR)$, $f \in C^\infty(C, \RR^p)$ and $\varphi \in \mathrm{Diff} (C)$. We call $A$ the {\em linear part} of the automorphism and $f$ its {\em translation part}. We define $\mathrm{Aut}_{E^q}^\ZZ (\RR^p \times C \to C)$ as the set of automorphisms of $\RR^p \times C \to C$ such that $A \in \mathrm{GL}_p (\ZZ)$ and $f$ composed with the projection onto $E^q$ parallel to $E^{p-q}$ is constant.

We identify the automorphisms of $(S^1)^p$ with the matrices of $\mathrm{GL}_p(\ZZ)$. An automorphism of the trivial $(S^1)^p$-principal bundle over $C$ can be written as $(S^1)^p \times C \ni (\bar a,x) \mapsto (A \bar a + \bar f (x), \varphi(x))$ where $A \in \mathrm{GL}_p (\ZZ)$, $\bar f \in C^\infty(C, (S^1)^p)$ and $\varphi \in \mathrm{Diff} (C)$. Since the manifold $C$ is simply-connected, any element of $\mathrm{Aut} ((S^1)^p \times C \to C)$ can be lifted to an element of $\mathrm{Aut} (\RR^p \times C \to C)$, uniquely defined up to a choice of base-points (we will omit this choice since it does not impact our results). We define $\mathrm{Aut}_{E^q} ((S^1)^p \times C \to C)$ as the set of automorphisms of the trivial $(S^1)^p$-principal bundle over $C$ whose lifts lie in $\mathrm{Aut}_{E^q}^\ZZ (\RR^p \times C \to C)$.
\end{definition}

With these definitions, we deduce from the previous discussion that $\pi_1(M)$ is a subgroup of $\mathrm{Aut}^\ZZ_{E^q} (\RR^p \times C \to C)$ and $\pi_1(M)/ \Gamma_0$ is a subgroup of $\mathrm{Aut}_{E^q} ((S^1)^p \times C \to C)$. This implies that, if we consider the projection
\[
\mathcal P : \mathrm{Aut}^\ZZ_{E^q} (\RR^p \times C \to C) \to \mathrm{Aut}_{E^q} ((S^1)^p \times C \to C),
\]
one has $\pi_1(M) = \mathcal P^{-1} (\pi_1(M)/ \Gamma_0)$ because $\ZZ^p \simeq \Gamma_0 \subset \pi_1(M)$.

We know from Remark~\ref{remark1} that $\pi_1(M) / \Gamma_0$ is a discrete group acting properly on $C$, so $N/ \bar P \simeq C/(\pi_1(M) / \Gamma_0)$ has an orbifold structure. Since $C$ is a simply connected manifold, this orbifold is {\em good}, and it is compact because, $M$ being compact, the projection onto $C$ of any compact fundamental domain for the co-compact action of $\pi_1(M)$ on $\tilde M$ is a compact fundamental domain for the action of $\pi_1(M) / \Gamma_0$ on $C$.

We assume from now on that $\pi_1(M)$ acts on ${\bar P}^0$ without torsion element, which is always possible up to a finite covering as explained in Section~\ref{finitecovering}. If we pick $\gamma \in \pi_1(M)$ such that $\bar \gamma \in \pi_1(M) / \Gamma_0$ acts as the identity on $C$, we know thanks to Remark~\ref{remark1} that $\gamma$ has finite order, thus $A_\gamma = \mathrm{I}_p$ because $A_\gamma$ cannot be a torsion element, and the translation part $f_\gamma$ has to be constant because it exists $k \ge 1$ such that $k f_\gamma \in \Gamma_0 \simeq \ZZ^p$. Hence $f_\gamma$ is an element of ${\bar P}^0$ and $\gamma \in \Gamma_0$ by definition. We conclude that the only elements of $\pi_1(M)$ whose induced action on $C$ is the identity are the elements of $\Gamma_0$. Consequently, the group $\pi_1(M)/\Gamma_0$ acts effectively on $C$, and it can be identified to the fundamental group of the good orbifold $N/\bar P$. Thus there is a short exact sequence
\begin{equation}
0 \to \ZZ^p \to \pi_1(M) \to \pi_1(N/\bar P) \to 0.
\end{equation}

Altogether, we have proved the following theorem:

\begin{theorem} \label{admissibledata}
Let $(M,c,D)$ be a simple LCP manifold. Then, there exists a simply connected manifold $C$ and an integer $p \ge 2$ such that $\tilde M \simeq \RR^p \times C$. There is a discrete subgroup $\Omega$ of $\mathrm{Aut}_{E^q} ((S^1)^p \times C \to C)$ such that, if we consider the canonical projection
\[
\mathcal P : \mathrm{Aut}^\ZZ_{E^q} (\RR^p \times C \to C) \to \mathrm{Aut}_{E^q} ((S^1)^p \times C \to C),
\]
one has $\pi_1(M) = \mathcal P^{-1}(\Omega)$, and the action of $\Omega$ restricted to $C$ is proper and co-compact. Moreover, there is a decomposition $\RR^p = E^q \oplus E^{p-q}$, such that $\dim E^q = q$, $E^q$ is the flat part of the LCP manifold, and the linear part (see Definition~\ref{linpart}) of $\Omega$ preserves this decomposition.

If in addition the LCP manifold is torsion-free, then the group $\Omega$ acts effectively on $C$, so it can be identified to the fundamental group of the compact good orbifold $C/\Omega$, and there is a short exact sequence $0 \to \ZZ^p \to \pi_1(M) \to \Omega \to 0$.
\end{theorem}

\begin{rem} \label{considernotorsion}
From the discussion of Section~\ref{finitecovering}, for any LCP structure $(M,c,D)$, up to considering a finite cover of $M$ one can assume that $\pi_1(M)$ acts by conjugation on $\RR^q \times {\bar P}^0 \simeq \RR^p$ without torsion. This means exactly that the linear part of $\Omega$ is a group with no torsion element and that the LCP structure is torsion-free. Then, up to a finite covering, $\Omega$ is isomorphic to $\pi_1(C/\Omega)$.
\end{rem}

\subsection{Construction of an LCP structure from admisible data} We now investigate the converse of the above statement, i.e. we consider a simply connected manifold $C$, an integer $p \ge 2$ and a discrete subgroup $\Omega$ of the automorphisms of $(S^1)^p \times C \to C$ satisfying the hypotheses of Theorem~\ref{admissibledata} and we will construct an LCP structure on $M := (\RR^p \times C) / \Omega$. More precisely, we assume that there exists a decomposition $\RR^p = E^q \oplus E^{p-q}$ with $E^q$ a $q$-dimensional subspace of $\RR^p$, $\Omega$ is a subgroup of $\mathrm{Aut}_{E^q} ((S^1)^p \times C \to C)$, the linear part of $\Omega$ preserves the decomposition $E^q \oplus E^{p-q}$ and its restriction to $E^q$ contains only similarities with respect to a given scalar product $g_{E^q}$, not all being isometries. We also assume that $\Omega$ restricted to $C$ acts properly and co-complactly (but not necessarily effectively), and for any element $\omega \in \Omega \setminus \{\mathrm{id}\}$ with a fix point $x \in C$, $\omega$ acts on $\RR^p \times \{ x \}$ (which is preserved by the transformation $\omega$) with no fixed point. This last condition is implicit in the conclusion of Theorem~\ref{admissibledata} since $\pi_1(M)$ acts freely on $\tilde M$.

\begin{rem}
In view of Remark~\ref{considernotorsion}, we could assume that the linear part of $\Omega$ is a group with no torsion element. In this case, for any $x \in C$ and $\omega \in \Omega \setminus \{\mathrm{id}\}$ in the isotropy group of $x$, this subgroup of $\Omega$ being finite because the action is proper, $\omega$ has finite order. Consequently, the linear part of $\omega$ is the identity, and the last condition listed above just means that the translation part does not vanish at $x$.
\end{rem}

The first step is to find a candidate for the fundamental group of the LCP manifold. For this, we introduce a group $G$ of transformations of $\tilde M := \RR^p \times C$. If
\[
\mathcal P : \mathrm{Aut}^\ZZ_{E^q} (\RR^p \times C \to C) \to \mathrm{Aut}_{E^q} ((S^1)^p \times C \to C)
\]
denotes the canonical projection, we define the subgroup $G$ of $\mathrm{Aut}^\ZZ_{E^q} (\RR^p \times C \to C)$ by
\begin{equation} \label{Gdefinition}
G := \mathcal P^{-1}(\Omega),
\end{equation}
which clearly contains $\ZZ^p$.

\begin{lemma} \label{actionG}
The group $G$ acts freely, properly discontinuously and co-compactly on $\RR^p \times C$.
\end{lemma}
\begin{proof}
According to \cite[Proposition 4.1]{FlaLCP}, since $\ZZ^p \trianglelefteq G$ by construction, $G$ acts freely and properly discontinuously on $\RR^p \times C$ if and only if $\ZZ^p$ acts freely and properly discontinuously on $\RR^p \times C$ and $G / \ZZ^p$ acts freely and properly discontinuously on $(\RR^p \times C)/ \ZZ^p \simeq (S^1)^p \times C$. The first claim is obvious, so we are left to check the second one.

Let $\omega \in \Omega$ which has a fix point $(\bar a, x) \in T^p \times C$. Then, the restriction of $\omega$ to $C$ satisfies $\omega (x) = x$, so $\omega$ is in the isotropy group of $x$, and by assumption $\omega\vert_{\RR^p \times \{ x \}}$ would have no fixed point if $\omega \neq \mathrm{id}$, thus $\omega = \mathrm{id}$ since $\omega$ has a fix point.

The projection map $(S^1)^p \times C \to C$ is proper, so the action of $G / \ZZ^p$ on $T^p \times C \to C$ is proper since $\Omega$ acts properly on $C$.

The co-compactness of the action is easily checked by choosing a compact elementary domain $D$ of $C$ for the action of $\Omega$ and considering $D' := [0,1]^p \times D$. Then one has $G (D') = \RR^p \times C$.
\end{proof}

From Lemma~\ref{actionG}, we know that $M := (\RR^p \times C) / G$ is a compact manifold. In order to define an LCP structure on $M$, it remains to construct a Riemannian metric $h$ on its universal cover $\tilde M :=  \RR^p \times C$, which is $G$-equivariant (in the sense that $G$ acts by similarities on $(\tilde M, h)$) and with reducible holonomy. This second point is the easiest: using the hypothesis made at the beginning of this section concerning the restriction of $\Omega$ to $E^q$ and writing $\tilde M = E^q \times (E^{p-q} \times C)$, the metric should be of the form $g_{E^q} + g_N$ where $g_{E^q}$ was introduced before, and $g_N$ is a metric on $N := E^{p-q} \times C$. The action of $G$ preserves the product structure, and its restriction to $E^q$ consists only on similarities, not all being isometries. Hence we can define the group homomorphism $\tilde\rho : G \to \RR^*_+$ which gives the similarity ratio of any element of this restriction.

We will now describe all possible Riemannian metrics for $g_N$ such that the group $E^{p-q}$ acts by isometries. For such a metric, the group $\ZZ^p$ acts by isometries on $\tilde M$ and $\tilde\rho$ descends to a group homomorphism $\rho : \Omega \to \RR^*_+$. These metrics are given in a basis adapted to the decomposition $E^{p-q} \times C$ by fields of matrices over $C$ of the form
\begin{align} \label{form2}
\left(
\begin{matrix}
Q & b_{FB} \\
b_{BF} & g_B
\end{matrix}
\right)
\end{align}
with $g_B$ being a Riemannian metric on $C$, $Q$ being a field of positive definite quadratic forms on $E^{p-q}$, and $b_{FB} : E^{p-q} \times T C \to \RR$ and $b_{BF} : T C \times E^{p-q} \to \RR$ are two bilinear forms related by the symmetry of the metric, i.e. $b_{FB}$ is determined by $b_{BF}$.

Let $\omega \in \Omega$ and we denote by $A \in \mathrm{GL}_p(\ZZ)$ its linear part and by $f \in C^\infty (C, E^{p-q})$ the $E^{p-q}$-component of its translation part. The representatives of $\omega$ in $G$ (i.e. the elements of $\mathcal P^{-1} (\{\omega\})$) all have the same differential, since they differ only by a translation element of $\ZZ^p$, thus the group $\Omega$ acts on $TN$ by push-forward. The restriction of $\omega_*$ to $E^{p-q}$ is a constant matrix corresponding to the linear part $A$ of $\omega$ restricted to $E^{p-q}$ and for any $X \in TC$, one has $\omega_* X = d \omega(X) + X(f)$. In particular, the transformation $\omega_*$ is $E^{p-q}$-invariant because $f$ is, and so is the action of $\Omega$ by push-forward.

The admssible metrics on $\tilde M$ should be $G$-equivariant, which is equivalent to the metric $g_N$ under the form \eqref{form2} being $\Omega$-equivariant, i.e. the admissible metrics correspond to the positive definite matrices satisfying:
\begin{align} \label{equivariance}
\omega^* \left(
\begin{matrix}
Q & b_{FB} \\
b_{BF} & g_B
\end{matrix}
\right) = \rho(\omega)^2 \left(
\begin{matrix}
Q & b_{FB} \\
b_{BF} & g_B
\end{matrix}
\right), && \forall \omega \in \Omega,
\end{align}
where the pull-back is well-defined by the previous discussion on the action of $\Omega$ by push-forward. We thus need to construct such an invariant metric. Using the fact that $\Omega$ acts co-compactly on $C$, there exists a compact $K \subset C$ such that $C = \Omega \cdot K$. By compactness, there is a finite cover $(U_i)_{i \in I}$ of $K$ by open sets whose closures are contained in charts of $C$ and are closed ball of the Euclidean space in these charts, so in particular each $U_i$ is relatively compact.  Defining $U = \cup_{i \in I} U_i$, it is easily seen that $U$ is a relatively compact open set such that $K \subset U \subset \bar U$. On each $U_i$ one can construct a Riemannian metric $g_i$ of the form \eqref{form2} (by taking $b_{BF}$ and $b_{FB}$ to be zero for example), and one can find a function $\chi_i : \tilde M \to \RR$ with support lying in $\bar U_i$ such that $\chi_i > 0$ on $U_i$. The metric $g := \sum\limits_{i \in I} \chi_i g_i$ is then of the form \eqref{form2} and is a Riemannian metric on $U$. We now define $g_N$ by the formula:
\begin{equation} \label{defgN}
g_N := \sum\limits_{\omega \in \Omega} \rho(\omega)^{-2} \omega^*  g.
\end{equation}

\begin{lemma} \label{lemmegN}
The metric $g_N$ given by Equation~\eqref{defgN} is well-defined and is a $\Omega$-invariant Riemannian metric on $E^{p-q} \times C$. Moreover, any $\Omega$-invariant metric arises from this construction.
\end{lemma}
\begin{proof}
First, we prove that $g_N$ is well-defined. It is sufficient to prove that the sum has only a finite number of non-zero terms on any sufficiently small neighborhood of a point in $C$.

Let $x \in C$. Since $C = \Omega \cdot \bar K \subset \Omega \cdot U$, there exists a small open subset $V$ of $U$ and $\omega_0 \in \Omega$ such that $x \in \omega_0 \cdot V$. For any $\omega \in \Omega$, the term $\omega^* g$ is not identically vanishing on $\omega_0 \cdot V$ only if $\omega \cdot \omega_0 \cdot V \cap \bar U \neq \emptyset$. The group $\Omega$ acts properly on $C$, so the set of elements $\omega' \in \Omega$ such that $\omega' \cdot \bar V \cap \bar U \neq \emptyset$ is finite because $\bar U$ is compact, thus the set of elements $\omega \in \Omega$ such that $\omega \cdot \omega_0 \cdot V \cap \bar U \neq \emptyset$ is finite. This implies that $g_N$ is a sum of a finite number of terms on $\omega_0\cdot V$, so it is well-defined and smooth on this neighborhood of $x$. This analysis holds for any point, hence $g_N$ is well-defined and smooth.

As a sum of positive-definite or null terms, $g_N$ is positive definite or null at any point of $C$. We need to prove that it is non-zero everywhere. But for any $x \in C$ there exists $\omega \in \Omega$ and $y \in U$ such that $\omega(y) = x$ and $\omega^* g (x)$ is non-zero. Thus $g_N$ is a Riemannian metric on $E^{p-q} \times C$.

We now check the equivariance property \eqref{equivariance}. Let $\omega_0 \in \Omega$. One has:
\[
\omega_0^* g_N = \omega_0^* \sum\limits_{\omega \in \Omega} \rho(\omega)^{-2} \omega^*  g = \rho(\omega_0)^2 \sum\limits_{\omega \in \Omega} \rho(\omega\cdot \omega_0)^{-2}\omega_0^* \omega^*  g = \rho(\omega_0)^2 \sum\limits_{\omega \in \Omega} \rho(\omega)^{-2} \omega^*  g = \rho(\omega_0)^2 g_N.
\]

To prove that any equivariant metric is obtained by this construction, we first consider a non-negative smooth function $\chi$ with support in $\bar U$ and such that $\chi > 0$ on $U$. Such a function exists due to the way we constructed $U$. For any $\omega \in \Omega$, let $\chi_\omega := \omega^* \chi$. With the same arguments we used for the metric $g_N$, one can prove that $\chi_T := \sum\limits_{\omega \in \Omega} \chi_\omega$ is well-defined, positive, smooth and $\Omega$-invariant. Now, let $g$ be any $\Omega$-equivariant metric, and we define
\begin{align*}
\forall \omega \in \Omega, && g_\omega := \frac{\chi_\omega}{\chi_T} g.
\end{align*}
Then for any $\omega \in \Omega$ one has $g_\omega := \omega^* g_\mathrm{id}$. This yields:
\begin{align*}
g = \sum\limits_{\omega \in \Omega} \frac{\chi_\omega}{\chi_T} g = \sum\limits_{\omega \in \Omega} \omega^* g_\mathrm{id}
\end{align*}
and $g$ is constructed in the same way as $g_N$.
\end{proof}

The previous discussion together  with Lemma~\ref{lemmegN} allows us to define an LCP structure on $M$ induced by the Riemannian structure $(\tilde M, g_{E^q} + g_N)$. The flat part of this LCP manifold contains $E^q$ and we can prove the following:

\begin{prop} \label{irreductibility}
If $E^q$ is exactly the flat part of the LCP structure, then the characteristic group of the LCP manifold is the smallest vector subspace $F$ of $\RR^p$ containing $E^q$ and generated by a subfamily of $\ZZ^p$. In particular, the projection of $E^q$ onto $\RR^p/\ZZ^p$ is dense in the projection of $F$.
\end{prop}
\begin{proof}
Since we are working on an LCP manifold, we will use the notations of Section~\ref{preliminaries} in this proof. A reasoning similar as the one of Section~\ref{finitecovering} allow us to consider, up to a finite covering, that the linear part of $\Omega$ has no torsion, without changing the characteristic group and the lattice $\Gamma_0$.

We first prove that the vector subspace $F$ introduced in the statement of this proposition is unique and well-defined. In order to do so, it is sufficient to prove that the intersection of two subspaces $F_1$, $F_2$ of $\RR^p$ containing $E^q$ and generated by subfamilies $\mathcal B_1$ and $\mathcal B_2$ of $\ZZ^p$ still has this property. The fact that $E^q \subset F_1 \cap F_2$ is obvious, so it remains to prove that $F_1 \cap F_2$ is generated by a subfamily of $\ZZ^p$. Denote by $F_1'$ and $F_2'$ the $\QQ$-subspaces of $\QQ^p$ generated by $\mathcal B_1$ and $\mathcal B_2$ respectively, so that $F_1 = F_1' \otimes_\ZZ \RR$ and $F_2 = F_2' \otimes_\ZZ \RR$. Then one has
\[
F_1 \cap F_2 = (F_1' \otimes_\ZZ \RR) \cap (F_2' \otimes_\ZZ \RR) = (F_1' \cap F_2') \otimes_\ZZ \RR,
\]
and any $\QQ$-basis of $F_1' \cap F_2'$ gives a basis with vectors in $\ZZ^p$ after multiplying each vector by a suitable integer. This induces a basis of $F_1 \cap F_2$ with elements in $\ZZ^p$.

The non-flat part of the LCP manifold is $N := E^{p-q}\times C$ since $E^q$ is exactly the flat part and $N$ is orthogonal to $E^q$. By definition, ${\bar P}^0$ is the connected component of the identity in $\bar P$, the closure of the restriction of $G$ to $N$.

Let $\omega \in \Omega$ such that $\omega \vert_N \in {\bar P}^0$. There exists a continuous path $\sigma : [0,1] \to {\bar P}^0$ such that $\sigma(0) = \mathrm{id}$ and $\sigma(1) = \omega \vert_N$. Let $x \in C$. The set $\sigma ([0,1])\vert_C (\{ x \})$ is a path-connected subset of $C$. Since the compact-open topology on metric spaces is characterized by the uniform convergence on compacts and the elements of ${\bar P}^0$ are in the closure of $G\vert_N$, the closure of $G\vert_N (\{ (0,x) \})$ (where $0 \in E^{p-q}$) must contain $\sigma ([0,1]) (\{ (0,x) \})$, so the closure of $G\vert_C (\{ x \})$ must contain $\sigma ([0,1])\vert_C (\{ x \})$. Yet, $G\vert_C (\{ x \}) = \Omega\vert_C (\{ x \})$ is a discrete subset of $C$ because $\Omega$ acts properly discontinuously on $C$. Thus, $\sigma ([0,1])\vert_C (\{ x \})$ is reduced to a single point, yielding $x = \sigma(0)\vert_C (x) \in \sigma ([0,1])\vert_C (\{ x \}) = \{ x \}$, and we deduce that $\sigma (1)\vert_C (x) = x$. It follows that $\sigma(1)\vert_C = \mathrm{id}$ and $\sigma(1) \in E^{p-q}$ because the linear part of $\Omega$ is has no torsion. Consequently, the elements of $G \cap (E^q \times {\bar P}^0) = \Gamma_0$ are translations of $\RR^p$ and there exists $m \in \NN$ such that $\Gamma_0 \subset \frac{1}{m} \ZZ^p$ because $G$ acts properly discontinuously on $\RR^p \times C$.

We can now prove that $F = E^q \times {\bar P}^0$. Indeed, the vector space $E^q \times {\bar P}^0$ admits $\Gamma_0 \subset \frac{1}{m} \ZZ^p$ as a full lattice, so it is generated by a subfamily of $\ZZ^p$, thus it contains $F$. On the other hand, $E^q / \Gamma_0$ has to be dense in $(E^q \times {\bar P}^0) / \Gamma_0$ by Property~\ref{density}, so it is dense in $(E^q \times {\bar P}^0) / (m\Gamma_0)$. In particular, if $F'$ is a subspace of $\RR^p$ generated by a subfamily of $\ZZ^p$ with $F \subset F'$, then $F / \ZZ^p$ and $F' / \ZZ^p$ are two sub-tori of $\RR^p / \ZZ^p$ with $F / \ZZ^p \subsetneq F' / \ZZ^p$, but the image of $E^q$ is contained in $F/\ZZ^p$, so it is not dense in $F' / \ZZ^p$. We conclude that $E^q \times {\bar P}^0 = F$ and the image of $E^q$ in $\RR^p/\ZZ^p$ is dense in the image of $F$ because it is dense in $(E^q \times {\bar P}^0)/\Gamma^0$.
\end{proof}

The whole discussion of this section is summarized in the following theorem:

\begin{theorem} \label{converse}
Let $C$ be simply connected manifold. Let $p \ge 2$ be an integer and $\Omega$ be a discrete subgroup of $\mathrm{Aut}_{E^q} ((S^1)^p \times C \to C)$ whose restriction to $C$ acts properly and co-compactly (i.e. $C/\Omega$ is a compact good orbifold). Assume that there exists a decomposition $\RR^p = E^q \oplus E^{p-q}$ with $\mathrm{dim} \ E^q = q$ preserved by the linear part of $\Omega$, and such that the restriction of this linear part to $E^q$ contains only similarities with respect to a given scalar product $g_{E^q}$, not all being isometries. We also assume that for any element $\omega \in \Omega \setminus \{ id \}$ with a fix point $x \in C$, $\omega\vert_{\RR^q \times \{x \}}$ has no fixed point. Then, considering the canonical projection (see Definition~\ref{linpart})
\[
\mathcal P : \mathrm{Aut}^\ZZ_{E^q} (\RR^p \times C \to C) \to \mathrm{Aut}_{E^q} ((S^1)^p \times C \to C),
\]
the group $G := \mathcal P^{-1}(\Omega)$ acts freely, properly and co-compactly on $\tilde M := E^q \times N$ where $N := E^{p-q} \times C$, and there exists a Riemannian metric $g_N$ on $N$ such that $G$ acts by similarities, not all being isometries on $(\tilde M, h := g_{E^q} + g_N)$. All such metrics $g_N$ are constructed as in Equation~\eqref{defgN}. This induces an LCP structure on $M := \tilde M/G$.
\end{theorem}

\section{Discussion on the hypotheses and further examples}

Theorem~\ref{admissibledata} together with Theorem~\ref{converse} give a description of simple LCP manifolds. However, there are still open questions concerning the hypotheses and possible simplifications that one could make in these statements. Indeed, the existence of the group $\Omega$ in Theorem~\ref{converse} is the only obstruction for the construction of an LCP manifold starting from a simply connected manifold $C$. In this section we discuss cases where one can construct such a group and we provide examples proving that some hypotheses cannot be removed. We will use the notations of these two theorems in the following section.

\subsection{The orbifold hypothesis and the structure of the fundamental group}

We begin this section by providing an example satisfying the hypotheses of Theorem~\ref{converse} where $C / \Omega$ is a compact orbifold which is not a manifold, i.e. $\Omega$ acts effectively on $C$ with fixed points.

\begin{example} \label{withorbifold}
Consider $C := S^2 \times \RR \subset \RR^3 \times \RR$ and the group $\Omega \simeq \ZZ/2\ZZ \times \ZZ$ of automorphisms of $(S^1)^2 \times C \to C$ acting for any $(\bar a,(x,y,z),s) \in (S^1)^2 \times S^2 \times \RR$, and for all $m \in \ZZ$ as:
\begin{align*}
&(\bar 1, 0) \cdot (\bar a, (x,y,z), s) = (\bar a + (0,1/2)^T,(-y,-x,z), s), \\
&(\bar 0, m) \cdot (\bar a, (x,y,z), s) = \left(\left( \begin{matrix} 1 & 2 \\ 2 & 3 \end{matrix} \right)^m \bar a, (x,y,z), s + m\right),
\end{align*}
i.e. $(\bar 1, 0)$ is the rotation of axis $z$ and angle $\pi$ on $S^2$. One easily check that this satisfies all the necessary hypotheses, since the matrix is diagonalizable with real eigenvalues different from $\pm 1$, and the elements $(\bar 1, m)$ for $m \in \ZZ$ have non-zero translation part. The orbifold $C/\Omega$ is not a manifold since the point $((0,0,1), 0)$ has a non-trivial isotropy group.
\end{example}

Example~\ref{withorbifold} proves that we cannot drop the case where $C/\Omega$ is an orbifold. However, in this example there is a finite covering (or equivalently a subgroup $\Omega'$ of $\Omega$ with finite index) of the LCP manifold constructed via Theorem~\ref{converse} such that $C/\Omega'$ is a manifold, since $C/\Omega$ is actually {\em very good}, being finitely covered by $S^2 \times S^1$. One can then ask whether for any LCP manifold it is possible to find a finite covering such that $C/\Omega$ is always a manifold in Theorem~\ref{admissibledata}. The answer is positive when $C/\Omega$ is the product of a manifold with a $2$-dimensional orbifold for  example, because any good compact $2$-orbifold is very good \cite[Theorem 2.5]{Scott}. We do not know if there exists a counter-example to this statement, so this question is still open.

Another natural question is the following: is the group $G$ defined in Theorem~\ref{converse} always a semi-direct product $\ZZ^p \rtimes \Omega$, as it is in all the examples we gave so far? The answer to this question is no, as shown by Example~\ref{notsemidirect} below.

\begin{example} \label{notsemidirect}
Consider $C := \RR^2$ and $\Omega\simeq \ZZ^2$ the group of automorphisms of $(S^1)^2 \times C \to C$ given by:
\begin{align}
(m_1, m_2) \cdot (\bar a, (x,y)) := (A^{m_1} \bar a + m_2 \tau, (m_1 x, m_2 y)), && \forall (\bar a, (x, y)) \in (S^1)^2 \times \RR^2,
\end{align}
where 
\begin{align*}
A := \left( \begin{matrix} -1 & 1 \\ 1 & -2 \end{matrix} \right), && \tau := \frac{1}{5} \left( \begin{matrix} 3 \\ 1 \end{matrix} \right).
\end{align*}
This group is well-defined and satisfies the hypotheses of Theorem~\ref{converse} since the matrix $A$ is diagonalizable with real eigenvalues different from $\pm 1$. Now, for the group $G$ defined by \eqref{Gdefinition} to be a semi-direct product, there should exist a subgroup $H$ of $G$ such that $H \cap \ZZ^2 = \{ \mathrm{id} \}$ and $H$ is isomorphic to $\Omega \simeq \ZZ^2$. However, for such a subgroup to exist, one should find representatives of $(1,0)$ and $(0,1)$ whose commutator is zero since $\Omega$ is Abelian. This is equivalent to finding $\tau_1, \tau_2 \in \ZZ^p$ so that the affine transformations of $\RR^2$ given by $(A,\tau_1), (I_2, \tau + \tau_2) \in \mathrm{GL}_2(\RR) \ltimes \RR^2$ commute. We have
\begin{align*}
(A,\tau_1)^{-1} \cdot (I_2, \tau + \tau_2)^{-1} \cdot (A, \tau_1) \cdot (I_2, \tau + \tau_2) = (I_2, (A - I_2) (\tau + \tau_2)) = (1,0)^T + (A - I_2) \tau_2.
\end{align*}
For the commutator to be zero, we should thus have $(1,0)^T = (I_2 - A) \tau_2$, and writing $\tau_2 =: (a,b)^T$ one obtains the system
\begin{align*}
\left\lbrace \begin{matrix} -2a + b = 1 \\ a -3b = 0 \end{matrix}\right.,
\end{align*}
which has no solution for $a, b \in \ZZ$. Thus, the group $G$ is not a semi-direct product of the form $\ZZ^p \rtimes \Omega$.
\end{example}

\subsection{The linear part of $\Omega$}
In order to understand which manifold can lead to the construction of an LCP manifold, it is important to know what are the possible group $\Omega$, and in particular what are their linear part. We study here the subgroups of $\mathrm{GL}_p(\ZZ)$ appearing as the linear part of $\Omega$. These groups should be finitely generated because the LCP manifold is compact, they should preserve a decomposition $E^q \oplus E^{p-q}$ of $\RR^p$ with $\mathrm{dim}(E^q) = q$, their restriction to $E^q$ consisting only of similarities for a given scalar product on $E^q$, but not all being isometries. Moreover, in regard of Proposition~\ref{irreductibility} together with Theorem~\ref{admissibledata} we can assume that the image of $E^q$ is dense in $\RR^p/\ZZ^p$ since we want to describe exactly the characteristic groups of LCP manifolds. Let $U$ be such a group.

The elements of $U$ have the following property, which is a consequence of the Jordan-Chevalley decomposition:

\begin{prop} \label{Ussimple}
All the elements of $U$ are semi-simple.
\end{prop}
\begin{proof}
Let $A \in U$. Since $\QQ$ is a perfect field, $A$ admits a Jordan-Chevalley decomposition $A =: D + N$ where $D \in M_p(\QQ)$ is semi-simple, $N \in M_p(\QQ)$ is nilpotent and $[D,N] = 0$. There exists an integer $m \in \NN$ such that $m D \in M_p(\ZZ)$ and $m N \in M_p(\ZZ)$. The linear transformations $m A = m D + m N$ and $m D$ descend to two group endomorphisms of the $p$-torus $\RR^p / \ZZ^p$, because they are matrices of $M_p(\ZZ)$.

Since $A\vert_{E^q}$ is a similarity, there exists $(\lambda, O, P) \in \RR^*_+ \times O(q) \times \mathrm{GL}_q (\RR)$ such that
\begin{equation} \label{Arestricted}
A\vert_{E^q} = \lambda P^{-1} O P
\end{equation}
and in particular, $A\vert_{E^q}$ is diagonalizable in $\CC$, then semi-simple. The Jordan-Chevalley decomposition of $A \vert_{E^q}$ is given by $D \vert_{E^q} + N \vert_{E^q}$, so we have $N \vert_{E^q} = 0$ because $A \vert_{E^q}$ is semi-simple, implying $A \vert_{E^q} = D \vert_{E^q}$. Consequently, the endomorphisms of $\RR^p / \ZZ^p$ induced by $m A$ and $m D$ coincide on the image of $E^q$ in $\RR^p / \ZZ^p$, which is dense, so they are equal by continuity. We conclude that $m A = m D$ and $A = D$, so $A$ is semi-simple.
\end{proof}

\begin{cor} \label{JCdecomp}
There exists a basis of $\RR^p$ containing only vectors of $\ZZ^p$ such that in this basis the matrix $A$ is diagonal by blocks, where the diagonal blocks have irreducible characteristic polynomial in $\ZZ$.
\end{cor}
\begin{proof}
We will work in the $\QQ$-vector space $\QQ^p$, in which $A$ is a well-defined linear transformation, and we will extend the result to $\RR^p$.

The characteristic polynomial $\chi_A$ of $A$ can be decomposed into monic irreducible factors over $\ZZ$ as
\begin{equation}
\chi_A = \prod_{k=1}^m P_k^{\alpha_k},
\end{equation}
where the polynomials $P_k$ are pairwise prime. This gives a decomposition of $\QQ^p$ into invariant subspaces
\begin{equation}
\QQ^p = \bigoplus_{k = 1}^m \mathrm{ker} \ P_k^{\alpha_k} := \bigoplus_{k = 1}^m \mathrm{ker} E_k,
\end{equation}
and the projections onto the $P_k^{\alpha_k}$ are polynomials in $A$ with coefficients in $\QQ$, so they are matrices with coefficients in $\QQ$. Let $1 \le k \le m$. The vector space $E_k$ admits as a basis any basis $\mathcal B_k$ of the full-lattice given by $\ZZ^p \cap E_k$. For a vector $v \in \mathcal B_k$, $A v \in \ZZ^p \cap E_k$ and then $A v$ is a linear combination of elements of $\mathcal B_k$ with coefficients in $\ZZ$. This means that the matrix of $A\vert_{E_k}$ written in $\mathcal B_k$ has coefficients in $\ZZ$. If we define the basis $\mathcal B$ of $\QQ^p$ as the concatenation of the bases $\mathcal B_k$, then $A$ is diagonal by blocks with coefficients in $\ZZ$.

It remains to look at what happens for the restriction to each $E_k$, so we can assume that $\chi_A = P^\alpha$ where $P$ is an irreducible polynomial in $\ZZ$. Since $A$ is semi-simple by Proposition~\ref{Ussimple}, $P$ is the minimal polynomial of $A$. The Frobenius decomposition gives the existence of a decomposition of the ambient vecto-space $\bigoplus_{k = 1}^\ell F_k$ where the $F_k$ are cyclic, stable by $A$ and the characteristic polynomial of $A \vert_{F_k}$ is $P$. The same argument as in the first part of the proof then gives us a basis of vectors with coefficients in $\ZZ$ adapted to the decomposition, in which $A \vert_{F_k}$ has coefficients in $\ZZ$ for each $k$.
\end{proof}

Let $A \in U$, and assume that $A\vert_{E^q}$ is not an isometry for the scalar product given on $E^q$, i.e. $\lambda \neq 1$ in equation \eqref{Arestricted}. Using Corollary~\ref{JCdecomp}, we can write $A$ under the form $\mathrm{Diag}(A1, \ldots A_m)$, where the blocks $A_k$ have irreducible characteristic polynomials. The subspace $E^q$  is spanned by real and complex part of complex eigenvectors of $A$ (since $E^q$ and $E^{p-q}$ are stable by $A$), i.e. it is a subspace of the sum of the eigenspaces of $A_1, \ldots, A_m$ whose associated eigenvalues have absolute value $\lambda$, which is stable by $A$. Moreover, the elements of $U$ all commute with the projector $\mathbf P$ on $E^q$ parallel to $E^{p-q}$.

Conversely, starting from a matrix $A \in \mathrm{GL}_p(\ZZ)$ of the form of Corollary~\ref{JCdecomp}, we can give a theoretical way of constructing an admissible group $U$. First, there should exist $\lambda \neq 1$ such that each block of $A$ has at least one eigenvalue of absolute value $\lambda$. Choose a $q$-dimensional subspace $E^q$ of $\RR^p$ stable by $A$, such that $A\vert_{E^q}$ has only eigenvalues with absolute value $\lambda$ and the image of $E^q$ in $\RR^p / \ZZ^p$ is dense (it is always possible, since we can just take the space spanned by the real and complex parts of eigenvectors with eigenvalues of absolute value $\lambda$). Since $A$ is semi-simple, one can choose a stable space $E^{p-q}$ supplementary to $E^q$, inducing a projector $\mathbf P$ on $E^q$ parallel to $E^{p-q}$. Let $\mathrm{Com} (\mathbf P) := \{ M \in \mathrm{GL}_p(\ZZ) \ \vert \ M \mathbf P = \mathbf P M \}$. Let $\Theta : \mathrm{Com} (\mathbf P) \to \mathrm{GL}_q(\RR)$, $M \mapsto \vert \det M\vert_{E^q} \vert^{-1} M\vert_{E^q}$ and $S := \Theta(\mathrm{Com} (\mathbf P))$. Remark that $\vert \det A\vert_{E^q} \vert^{-1} A\vert_{E^q}$ has only eigenvalues with absolute value $1$, so it is contained into a maximal compact subgroup $H$ of $\mathrm{GL}_q (\RR)$. This subgroup $H$ is conjugated to $O(q)$ and then defines a scalar product on $E^q$. Taking $U := \Theta^{-1} (H)$ defines an admissible group, and any admissible group can be obtained by this construction, up to a choice of $\mathbf P$, $H$ and taking a subgroup of $U$.

\subsection{The non-constant translations}
A visible difference between the group $\Omega$ introduced in Theorem~\ref{converse} and all the examples we gave so far is the presence of a possibly non-constant translation part in an element of $\Omega$. Nevertheless, one can easily obtain non-constant translations from any example with constant translation part introduced before. For that, it is sufficient to consider the trivial bundle $\RR^p \times C$ in Theorem~\ref{converse} and the diffeomorphism
\begin{equation}  \label{transfo}
\varphi : \RR^p \times C \to \RR^p \times C, (a,x) \mapsto (a + s(x), x)
\end{equation}
where $s : C \to E^{p-q}$ is any smooth function. Now, if one takes an element $\omega \in \Omega$ which has $\tilde \omega = (\RR^p \times C \ni (a,x) \mapsto (A a + b, \omega\vert_C (x)))$ as a lift, it follows:
\begin{align*}
\varphi^{-1} \tilde \omega \varphi (a,x) = (A a + b + A s(x) - s(\omega\vert_C (x)), \omega\vert_C (x)), && \forall (a,x) \in \RR^p \times C,
\end{align*}
and taking a function $s : C \to E^{p-q}$ such that $s$ is not equivariant will lead to non-constant translations. The new LCP structure is then isomorphic to the modified one, and one can then ask wether all the examples with non-constant translations arise from this construction. In other words, is it always possible to find a function $s : C \to E^{p-q}$ such that for any $\omega \in \Omega$ with $\tilde \omega = (\RR^p \times C \ni (a,x) \mapsto (A a + f(x),\omega\vert_C (x))$ as a lift, there exists a constant $c \in \RR^p$ such that
\begin{align} \label{condition}
\tilde \omega (s(x), x) = (s(\omega\vert_C (x)) + c, \omega\vert_C (x)), && \forall x \in C.
\end{align}
Indeed, using the diffeomorphism $\varphi$ defined in \eqref{transfo} one has in this case:
\begin{align*}
\varphi^{-1} \tilde \omega \varphi (a,x) = (A a + c, \omega\vert_C (x)), && \forall (a,x) \in \RR^p \times C,
\end{align*}
and all the translations are constant. Conversely, if there exists a diffeomorphism $\varphi$ and a section $s$ as in \eqref{transfo} such that through this transformation the translation part of $\Omega$ only contains constants, then $s$ has to satisfy the condition given by Equation~\eqref{condition}. Such a function $s$ does not always exists, as shown by the following example:

\begin{example}
We define the matrix
\begin{equation} \label{A0}
A_0 := \left( \begin{matrix} 1 & 1 \\ 1 & 2 \end{matrix} \right).
\end{equation}
This matrix is diagonalizable with eigenvalues $\lambda := \frac{3 + \sqrt{5}}{2}$ and $\lambda^{-1}$. The matrix of eigenvectors is
\begin{align*}
P := \left( \begin{matrix} \frac{1+\sqrt{5}}{2} & \frac{1 - \sqrt{5}}{2} \\ 1 & 1 \end{matrix} \right).¨
\end{align*}
We consider the affine transformations of $\RR^4$ depending on a parameter $z \in \RR$ given by:
\begin{equation}
\begin{aligned}
\alpha &: X \mapsto \left( \begin{matrix} A_0 & 0 \\ 0 & A_0  \end{matrix} \right) X, &&
\beta_1(z) &: X \mapsto X + (0,0,z,0), &&
\beta_2(z) &: X \mapsto X + (0,0,0,z).
\end{aligned}
\end{equation}
These maps canonically descend to $(S^1)^4$. One easily checks that the commutators of these transformations are:
\begin{equation}
\begin{aligned} \label{commutators1}
[\alpha,\beta_1(z)] = \beta_1(z)^{-1} \beta_2(z), && [\alpha,\beta_2(z)] = \beta_1(z), && [\beta_1(z), \beta_2(z)] = \mathrm{id}.
\end{aligned}
\end{equation}

We consider the manifold $C_0 := \RR^2 \times \RR^*_+$ on which acts co-compactly, freely and properly the group
\[
H := \langle T : (y,t) \mapsto (A_0 y, \lambda t), T_1 : (y,t) \mapsto (y+(1,0)^T, t), T_2 : (y,t) \mapsto (y+(0,1)^T, t) \rangle,
\]
where $(y,t) \in \RR^2 \times \RR$. Note that $C_0/ H$ can actually be given the structure of an LCP manifold. The universal cover of the compact manifold $C_0/H \times S^1$ is $C := C_0 \times \RR$, and its fundamental group is $H \times \ZZ$. We have the relations
\begin{equation}
\begin{aligned} \label{commutators1}
[T,T_1] = T_1^{-1} T_2, && [T,T_2] = T_1, && [T_1, T_2] = \mathrm{id},
\end{aligned}
\end{equation}
so for any $z \in \RR$ the groups $H$ and $\langle \alpha, \beta_1(z), \beta_2 (z) \rangle$ are isomorphic via the isomorphism $\iota_z$ determined by $\iota_z (T) = \alpha$, $\iota_z (T_j) = \beta_j(z)$. We consider the subgroup $\Omega \simeq H \times \ZZ$ of $\mathrm{Aut} ((S^1)^4 \times C \to C)$ given by
\begin{align*}
(h, n) \cdot (\bar a,(y,t)) = (\iota_z(h) \bar a, h (y, t), z), && \forall (h, n) \in H \times \ZZ, \forall (\bar a, (y,t), z) \in (S^1)^4 \times C_0 \times \RR.
\end{align*}
In order to prove that this group is well-defined, it is sufficient to show that the commutators of the generators $(T, 0), (T_1, 0), (T_2, 0), (\mathrm{id}, 1)$ of $\Omega$ satisfy the suitable relations. It is easily seen that
\begin{align*}
[(T,0) \cdot , (T_1,0) \cdot ] = (T_1^{-1} T_2,0) \cdot, && [(T,0) \cdot, (T_2,0) \cdot] = (T_1,0) \cdot, && [(T_1,0) \cdot, (T_2,0) \cdot] = (\mathrm{id}, 0) \cdot
\end{align*}
because these commutators can be computed at a fixed $z \in \RR$ and then correspond to the relations \eqref{commutators1}. We compute the remaining commutators. For any $(\bar a, (y,t), z) \in (S^1)^4 \times C_0 \times \RR$ one has
\begin{align*}
&[(T,0) \cdot, (\mathrm{id},1) \cdot] (\bar a, (y,t), z)= (\bar a, (y,t), z) \\
&[(T_1,0) \cdot, (\mathrm{id},1) \cdot] (\bar a, (y,t), z)= (\bar a + (0,0,1,0)^T, (y,t), z) = (\bar a, (y,t), z)\\
&[(T_2,0) \cdot, (\mathrm{id},1) \cdot] (\bar a, (y,t), z)= (\bar a + (0,0,0,1)^T, (y,t), z) = (\bar a, (y,t), z),
\end{align*}
then all these commutators are equal to the identity, proving that $\Omega$ is well-defined.
Now, the linear part of $\alpha$ is a matrix $A$, which is diagonalizable under the form $\mathrm{Diag} (\lambda, \lambda^{-1},\lambda, \lambda^{-1})$. We choose a new coordinate system on $\RR^4 \times C = \RR^4 \times \RR^2 \times \RR^*_+ \times \RR$ in the following way: if $(x_1, x_2, y_1, y_2, r_1, r_2, t, z)$ is the canonical system of coordinates we set
\[
\left( \begin{matrix} u_1 \\ u_2 \\ v_1 \\ v_2 \\ w_1 \\ w_2 \\ t \\ z \end{matrix} \right) = \left( \begin{matrix} P^ {-1} &  & & & \\ & P^{-1} & & & \\ & & P^{-1} & & \\ & & & 1 & \\ & & & & 1 \end{matrix} \right) \left( \begin{matrix} x_1 \\ x_2 \\ y_1 \\ y_2 \\ r_1 \\ r_2 \\ t \\ z \end{matrix} \right)
\]
In these new coordinates, $A$ is diagonal and we define $\frac{\partial}{\partial u} := \pi \frac{\partial}{\partial u_1} + \frac{\partial}{\partial v_1}$, so that the image of $\Span(u)$ under the canonical projection is dense in the torus $\RR^4/\ZZ^4$ (written in the old coordinates). The matrix $A$ preserves the decomposition $\RR^4 = \Span(\frac{\partial}{\partial u}) \oplus \Span(\frac{\partial}{\partial u_2}, \frac{\partial}{\partial v_1},\frac{\partial}{\partial v_2})$. This construction satisfies all the hypotheses of Theorem~\ref{converse}, so we obtain a group $G$ such that $(\RR^4 \times C)/G$ can be given an LCP structure with flat part containing $\Span(\frac{\partial}{\partial u})$. In particular, $\RR^4$ is contained in the characteristic group by Proposition~\ref{irreductibility}.

We consider the lift of $(T_1, 0) \in \Omega$ and $(\mathrm{id},1) \in \Omega$ respectively given by
\[
\tilde \omega_1 = (\RR^4 \times C_0 \times \RR \ni (a,(y,t),z) \mapsto (\beta_1(z) (a),T_1 (y,t), z)) \in G
\]
and
\[
\tilde \omega = (\RR^4 \times C_0 \times \RR \ni (a,(y,t),z) \mapsto (a, (y,t), z + 1)) \in G.
\]
If there existed a section $s$ satisfying \eqref{condition}, then using the transformation $\varphi$ given by \eqref{transfo}, one would get that
\begin{align*}
\varphi \tilde\omega_1 \varphi^{-1} &= (\RR^4 \times C_0 \times \RR \ni (a,(y,t),z) \mapsto (a + c_1,T_1 (y,t), z)) \\
\varphi \tilde\omega \varphi^{-1} &= (\RR^4 \times C_0 \times \RR \ni (a,(y,t),z) \mapsto (a + c_2, (y,t), z + 1))
\end{align*}
for $c_1,c_2 \in \RR$, so these two elements would commute. But one has
\[
[\varphi  \tilde\omega_1 \varphi^{-1}, \varphi \tilde\omega \varphi^{-1}] = \varphi[  \tilde\omega_1, \tilde\omega] \varphi^{-1} = (\RR^4 \times C_0 \times \RR \ni (a,(y,t),z) \mapsto (a + (0,0,1,0)^T, (y,t), z)),
\]
which is a contradiction. We can give an explicit metric on $\RR^4 \times C$ which defines an LCP structure: this metric is written in the coordinate system $(u_1, u_2, v_1, v_2, w_1, w_2, t, z)$ as
\begin{equation}
\left( \begin{matrix}
\frac{2}{\pi^2} & 0 & -\frac{1}{\pi} & 0 & 0 & 0 & 0 & \frac{w_1}{\pi} \\
0 & t^4 & 0 & 0 & 0 & 0 & 0 & 0 \\
-\frac{1}{\pi} & 0 & 1 & 0 & 0 & 0 & 0 & - w_1 \\
0 & 0 & 0 & t^4 & 0 & 0 & 0 & -t^4 w_2 \\
0 & 0 & 0 & 0 & 1 & 0 & 0 & 0 \\
0 & 0 & 0 & 0 & 0 & t^4 & 0 & 0 \\
0 & 0 & 0 &0 & 0 &0 & 1 & 0 \\
\frac{w_1}{\pi} & 0 & -w_1 & - t^4 w_2 & 0 & 0 & 0 & t^2 + w_1^2 + t^4 w_2^2
\end{matrix} \right).
\end{equation}
An orthonormal basis for this metric is given by:
\[
\pi \frac{\partial}{\partial u_1} + \frac{\partial}{\partial v_1}, t^{-2} \frac{\partial}{\partial u_2}, \frac{\partial}{\partial v_1}, t^{-2} \frac{\partial}{\partial v_2}, \frac{\partial}{\partial w_1}, t^{-2} \frac{\partial}{\partial w_2}, \frac{\partial}{\partial t}, t^{-1} \left( w_1 \frac{\partial}{\partial v_1} + w_2 \frac{\partial}{\partial v_2} + \frac{\partial}{\partial z} \right)
\]
and the dual frame is
\[
\frac{1}{\pi} d u_1, t^2 d u_2, d v_1 - \frac{1}{\pi} d u_1, t^2 (d v_2 - w_2 d z), d w_1, t^2 d w_ 2, d t, t d z.
\]
Straightforward computations show that the flat part of this LCP structure is exactly the integral manifold of the distribution spanned by $\frac{\partial}{\partial u}$. In addition, one has
\[
\left[ \frac{\partial}{\partial w_1}, w_1 \frac{\partial}{\partial v_1} + w_2 \frac{\partial}{\partial v_2} + \frac{\partial}{\partial z} \right] = \frac{\partial}{\partial v_1},
\]
showing that the distribution orthogonal to the fibers of $\RR^4 \times C \to C$ is not integrable, and in particular it is not possible to find a metric of the form \eqref{form2} with $b_{FB} = 0$.
\end{example}

However, it is always possible to remove the non-constant translation part when $G$ is a semi-direct product:

\begin{prop} \label{semidirectcase}
Assume that in Theorem~\ref{converse} $G = \ZZ^p \rtimes \Omega$ and $\Omega$ acts freely on $C$. Then the translation part $\Omega$ can be assumed to contain only constants belonging to $E^q$ up to a diffeomorphism.
\end{prop}
\begin{proof}
The assumption $G = \ZZ^p \rtimes \Omega$ allows us to define a subgroup $\tilde \Omega$ of $G$ inducing a splitting of the short exact sequence $0 \to \ZZ^p \to G \to \Omega \to 0$. In particular, there is an isomorphism $\iota: \Omega \to \tilde \Omega$.

From the previous discussion, it is sufficient to find a function $s : C \to E^{p-q}$ which is $\Omega$-invariant, i.e. which satisfies Equation~\ref{condition} with $c \in E^q$ for any $\omega \in \Omega$. Since the linear part of $\Omega$ preserves the decomposition $E^q \oplus E^{p-q}$, it is sufficient to find $s$ such that 
\begin{align} \label{condition2}
\iota(\omega)\vert_{E^{p-q} \times C} (s(x), x) = (s(\omega\vert_C (x)), \omega\vert_C (x)), && \forall \omega \in \Omega.
\end{align}
Consider the associated bundle
\[
\mathcal B = C \times_\Omega E^{p-q},
\]
where $\Omega$ acts on $C \times E^{p-q}$ by its natural action. This is a bundle over $C/\Omega$, which is a compact manifold since $\Omega$ acts freely on $C$. Its typical fiber $E^{p-q}$ is contractible, so it has a global smooth section \cite[Corollary 29.3]{Stee}. By a standard result, the sections of $\mathcal B$ are in one-to-one correspondence with the equivariant maps satisfying \eqref{condition2}, which implies the existence of the map $s$ we were searching for.
\end{proof}

\subsection{Existence of the group $\Omega$.} One question remains: when can one construct a group $\Omega$ as in Theorem~\ref{converse}? We know that $C/\Omega$ is a good compact orbifold since $\Omega$ acts properly and co-compactly on $C$. Thus, a way to answer this question is, starting from a good compact orbifold $\bar C = C / \pi_1(\bar C)$ (where $C$ is the universal cover of $\bar C$), to check if we can lift $\pi_1(\bar C)$ to a subgroup $\Omega$ of the automorphisms of a trivial principal torus bundle over $C$. We give here a necessary condition for the existence of $\Omega$, which turns out to be sufficient when $\bar C$ is manifold.

\begin{prop} \label{existencemorph}
Let $C$ be a simply-connected manifold and $\Omega$  be a group as given in the statement of Theorem~\ref{converse}. then $\Omega$ is isomorphic to a semi-direct product $\Omega' \rtimes \ZZ$ where $\Omega'$ is a subgroup of $\Omega$. 

Conversely, if $\bar C$ is a compact manifold with universal cover $C$ such that $\pi_1(\bar C) \simeq H \rtimes \ZZ$ for a subgroup $H$ of $\pi_1(\bar C)$, then there exists an integer $p \ge 2$ and a group $\Omega \subset \mathrm{Aut} ((S^1)^p \times C \to C)$ as in Theorem~\ref{converse}.
\end{prop}
\begin{proof}
Let $\rho : \Omega \to \RR^*_+$ associating to $\omega \in \Omega$ the similarity ratio of its linear part restricted to $E^q$. The group $\rho (\Omega)$ is a subgroup of $\RR^*_+$ with finite rank, generated by a finite number of independent elements. Let $\lambda$ be one of these elements and let $\pi_1: \rho (\Omega) \to \langle \lambda \rangle$ be the canonical projection. One has a short exact sequence
\[
0 \to \mathrm{ker} (\pi_1 \circ \rho) \to \Omega \to \langle \lambda \rangle \simeq \ZZ \to 0,
\]
and this sequence splits because one can find a section consisting of the map sending $\lambda$ to an element $\omega \in \Omega$ such that $\rho (\omega) = \lambda$. Thus $\Omega \simeq \mathrm{ker} (\pi_1 \circ \rho) \rtimes \ZZ$.

We now prove the converse part. We can take $p = 2$ and define $\Omega \simeq H \rtimes \ZZ$ by $\omega = \mathrm{id}$ for $\omega \in H$ and $\omega = ((S^1)^p \times C \ni (\bar a, x) \mapsto (A \bar a, \omega(x)))$ for $\omega \in \ZZ$, where $A$ is any suitable matrix (take for example the matrix defined in \eqref{A0}).
\end{proof}

\begin{rem}
if $C/\Omega$ is a manifold in Theorem~\ref{admissibledata}, we can actually say more. Indeed, since the fibration $\RR^p \times C \to C$ is a Riemannian fibration, the metric $h$ on $\RR^p \times C$ induces a metric $g_B$ on $C$ (we already emphasize this in this text). The fundamental group of the compact manifold $C/\Omega$ acts by similarities, not all being isometries on $(C, g_B)$, so $g_B$ is a non-Riemannian similarity structure on $C/\Omega$. An application of Theorem~\ref{LCPTh} gives that $(C, g_B)$ is either flat (a case which is classified \cite{Frie}), irreducible, or an LCP manifold. The two main conjectures remaining are then the following: are all LCP manifolds simple? and given an LCP manifold, can we always say, up to a finite cover, that $C/\Omega$ is a manifold? This last problem was tackled in section~\ref{withorbifold}. If the answers to these two questions are positive, then one could again decompose $(C, g_B)$ using Theorem~\ref{admissibledata}, and continue this process. This would end in a finite number of iterations, because the flat part is of positive dimension, leaving us with a flat or an irreducible manifold at the end.
\end{rem}

\renewcommand{\refname}
\end{document}